\documentclass[10pt]{article}

\usepackage{listings}
\usepackage[colorlinks=false, backref]{hyperref}

\usepackage{amsmath,amsthm,amsfonts,amssymb,amscd}
\usepackage{mathabx}
\usepackage{float}
\usepackage{graphicx}
\usepackage[usenames]{color}
\usepackage{makeidx}

\usepackage[mathscr]{euscript}
\usepackage{stmaryrd}
\usepackage{microtype}
\usepackage{booktabs}
\usepackage{cleveref}
\usepackage{bookmark}
\usepackage{mathrsfs}

\usepackage{emptypage}
\usepackage{tikz}

\DeclareMathAlphabet{\mathpzc}{OT1}{pzc}{m}{it}

\theoremstyle{plain}
\newtheorem{theorem}{\scshape Theorem}

\newtheorem{lemma}[theorem]{\scshape Lemma}
\newtheorem{corollary}[theorem]{\scshape Corollary}

\theoremstyle{definition}

\newtheorem{remark}{\scshape Remark}

\theoremstyle{definition}

\def\bbN{{\mathbb N}}

\def\bbR{{\mathbb R}}

\def\I{{\mathcal I}}

\def\K{{\mathcal K}}

\def\P{{\mathcal P}}

\def\R{{\mathcal R}}

\def\U{{\mathcal U}}
\def\V{{\mathcal V}}

\def\rE{{\rm E}}

\def\rO{{\rm O}}

\def\rT{{\rm T}}

\def\rk{{\rm k}}

\def\f{\text{\bf\emph{f}}}
\def\g{\text{\bf\emph{g}}}

\def\p{\text{\bf\emph{p}}}

\def\u{\text{\bf\emph{u}}}

\def\w{\text{\bf\emph{w}}}

\def\mC{{\mathscr C}}

\def\Varphi{{\boldsymbol\varphi}}

\definecolor{grey}{rgb}{0.5,0.5,0.5}
\definecolor{lightgrey}{rgb}{0.9,0.9,0.9}
\definecolor{darkgreen}{rgb}{0,0.6,0}
\definecolor{orange}{rgb}{1,0.5,0}

\DeclareMathOperator{\loc}{loc}

\def\curl{{\operatorname{curl}}}

\def\supp{{\operatorname{spt}}}

\def\hol{H^1_{\loc}( \mathbb{R}^2  ) }
\def\htl{H^2_{\loc}( \mathbb{R}^2  ) }

\def\cls #1{\overline{#1}}

\def\id{{\text{\rm Id}}}

\def\cptsubset{\hspace{1pt}{\subset\hspace{-2pt}\subset}\hspace{1pt}}

\def\qed{\hfill{$\square$}}

\def\bp{{\overline{\partial}\hspace{1pt}}}
\def\p{{\partial\hspace{1pt}}}

\def\n{{\rm n}}
\def\loc{{\operatorname{loc}}}
\def\cptspt{{c}}

\def\Forall{\forall\hspace{2pt}}

\def\Rn{{\bbR^{\hspace{0.2pt}\n}}}

\def\comm#1#2{{\llbracket#1,#2\rrbracket}}
\def\comm#1#2{{\{ #1,#2 \} }}

\def\({{(\hspace{-2pt}(}}
\def\){{)\hspace{-2pt})}}

\def\smallexp#1{{\text{\small #1}}}

\def\jump#1{{[\hspace{-2pt}[#1]\hspace{-2pt}]}}
\def\bigjump#1{{\big[\hspace{-3pt}\big[#1\big]\hspace{-3pt}\big]}}
\def\Bigjump#1{{\Big[\hspace{-4.5pt}\Big[#1\Big]\hspace{-4.5pt}\Big]}}

\def\cvl{{\,\convolution\,}}
\def\novertwo{\text{\small $\displaystyle{}\frac{\n}{2}$}}

\def\XXint#1#2#3{{\setbox0=\hbox{$#1{#2#3}{\int}$}
\vcenter{\hbox{$#2#3$}}\kern-.5\wd0}}

\usepackage{fancyhdr}
\title{Regularity of the velocity field for Euler vortex patch evolution}

\author{
Daniel Coutand
\\Department of Mathematics,
\\Heriot-Watt University
\\Edinburgh, EH14 4AS  UK
\\{\footnotesize email: d.coutand@ma.hw.ac.uk}
\and Steve Shkoller
\\Department of Mathematics
\\University of California
\\Davis, CA 95616 USA
\\{\footnotesize email: shkoller@math.ucdavis.edu}
}
\date{October 14, 2015}

\begin{document}

\maketitle

\noindent
{\bf Abstract.} {\small  We consider the vortex patch problem for both the 2-D and 3-D incompressible Euler equations.
 In 2-D, we prove that for vortex patches with $H^{k-0.5}$ Sobolev-class contour regularity, $k \ge 4$,
the velocity field  on both sides of the vortex patch boundary has $H^k$ regularity for all time.   In 3-D, we establish existence of solutions
to the vortex patch problem on a finite-time interval $[0,T]$, and
 we simultaneously establish the 
$H^{k-0.5}$ regularity of the two-dimensional  vortex patch boundary, as well as the $H^k$ regularity of the velocity fields on both sides of vortex patch boundary, for $k \ge 3$.}

\tableofcontents

\section{Introduction}
\subsection{The incompressible Euler equations}
Global existence for the Euler 2-D vortex patch problem was  first established by Chemin \cite{Chemin1993, Chemin1995}, Bertozzi \& Constantin \cite{BeCo1993}, and Serfati \cite{Serfati}; see also the recent article by Bae \& Kelliher \cite{BaKe2014}.   Local existence for
the 3-D vortex patch problem was proved by Gamblin \& Saint Raymond \cite{GaSa1995}.

 We are interested in the regularity
properties of the velocity field associated to the vortex patch evolution.
In particular, we analyze the   incompressible Euler equations on $ \mathbb{R}^\n  $, $\n=2,3$,  written as
 \begin{subequations}
 \label{euler}
\begin{align} 
u_t + \nabla_uu +\nabla p & =0 \,,\\
\operatorname{div} u & =0 \,,
\end{align} 
 \end{subequations}
 where $u(x,t)$ is the velocity vector field and $p(x,t)$ is  the pressure function, where the advection term $\nabla_uu$ denotes
  $\sum_{j=1}^\n  \frac{\partial u}{\partial x_j} u^j$.
%
 
\subsection{The 2-D vortex patch problem}\label{2dpatch}
Letting $\nabla ^\perp = (- \partial_{x_2}, \partial_{x_1})$, we define the 2-D vorticity function 
 $\omega (x,t) =  \nabla ^\perp \cdot  u(x,t) = u^2,_1 - u^1,_2$.  The vorticity $\omega$ is transported and satisfies
 \begin{equation}\label{vorticity}
 \omega _t + \nabla_u \omega =0 \,.
\end{equation} 
 Letting $\psi(x,t)$ denote the stream function, given by $u = \nabla ^\perp \psi$, we have that
 $\Delta \psi = \omega$,
 so that $ \psi(x,t) = {\frac{1}{2\pi}} \int_{\mathbb{R}^2  } \log|x-y| \omega (y) dy$.
Thanks to the Biot-Savart kernel $K(x) = {\frac{1}{2\pi}} \nabla ^\perp \log|x|$,
 \begin{equation}\label{BS}
u(x,t) = \int_{ \mathbb{R}^2  } K(x-y) \omega (y) dy \,.
\end{equation}

For each time $t \in [0, \infty )$, 
 let $\Omega^+(t)$ denote an open, simply-connected, and bounded subset of $ \mathbb{R}^2  $ with boundary
 $\Gamma(t):=\partial \Omega^+(t)$ given by a closed curve which is
 diffeomorphic to the circle $ \mathbb{S}  ^1$.   Let $\Omega^-(t)$ denote 
 $\overline{\Omega^+(t)}^ c$.  
The {\it 2-D vortex patch problem} consists of the following initial data for the Euler equations:
\begin{equation}\label{vp2d}
\omega _0(x) =\left\{
\begin{array}{ll}
1\,, & x \in \overline{\Omega^+(0)} \\
0\,, & x \in \Omega^-(0)\,.
\end{array}
\right.
\end{equation} 
The time-dependent open set $\Omega^+(t)$  is thus termed the {\it vortex patch}; the vortex patch boundary $\Gamma(t):=\partial \Omega^+(t)$ moves 
with the velocity of the fluid,
given by
$
u(x,t) = \int_{ \Omega^+(t) } K(x-y) dy $.  It follows that
\begin{equation}\label{problematic}
\nabla u(x,t) =  \int_{ \Omega^+(t) }  \nabla K(x-y) dy \,.
\end{equation} 

Given an initial 2-D vortex patch boundary $\Gamma(0)$ of H\"{o}lder class $\mC^{k, \alpha }$, it was established by Chemin \cite{Chemin1993}
and Bertozzi \& Constantin \cite{BeCo1993} that a unique solution exists for all time,  that the $\mC^{k, \alpha }$ contour regularity propagates,
 and that the gradient of  the velocity 
remains bounded for all time.
Their  proof of $\mC^{k, \alpha }$ contour regularity (in 2-D) can also be used to establish  $H^k$ contour regularity (we provide a
proof for the $\n$-dimensional case, $\n=2$ or $3$ in Section \ref{sec5}), 
and we state one of their fundamental results as
follows: Given an initial vortex patch boundary $\Gamma(0)$ of 
class $H^{k-0.5}$, $k \ge 3$, for all $t\in [0, \infty )$, there  exists a unique solution to the vortex patch problem, with non self-intersecting
boundary
$\Gamma(t)$, and satisfying the following estimate:
\begin{align} 
\label{bounds}
\frac{1}{|z|_*(t)}+ \| z( \cdot , t)\|_{H^{k-0.5}( \mathbb{S}  ^1)} + \| \nabla u ( \cdot , t) \|_{ L^ \infty ( \mathbb{R}^2  )}  & \le F(t) \,,
\end{align} 
where $z(\cdot ,t): \mathbb{S}  ^1 \to \Gamma(t)$ denotes an $H^{k-0.5}$-class parameterization of the vortex patch boundary 
$\Gamma(t)$, 
\begin{equation}\label{chordarc}
|z|_*(t) = \inf_{\theta_1\neq \theta_2} \frac{ | z(\theta_1, t) - z(\theta_2,t)| }{|\theta_1 - \theta_2|} \,,
\end{equation} 
 and $0< F(t) < \infty $ for any $t< \infty $.      We see that (\ref{bounds}) provides a strictly positive lower-bound on $|z|_*(t)$ which, in turn,
 provides a strictly positive lower bound for the metric $|\partial_\theta z(\theta)|$ and ensures that $\Gamma(t)$ does not self-intersect (see, for example,
 Majda \& Bertozzi \cite{MaBe2002}).  We identity $ \mathbb{S}  ^1$ with the interval $[0, 2 \pi]$.

\subsection{The 3-D vortex patch problem}\label{3dpatch}
 In three space dimensions, the 3-D vorticity $ \omega  = \operatorname{curl} u$ is a vector field,
and satisfies the 
vector equation
 \begin{equation}\label{vorticity3d}
 \omega _t + \nabla_u \omega =  \nabla _ \omega u \,,
\end{equation} 
where in components and for each $i=1,2,3$,  $[\nabla_u \omega]^i =\sum_{j=1}^3 \frac{ \partial \omega^i}{ \partial x_j} u^j$ and 
 $[\nabla_ \omega u ]^i =\sum_{j=1}^3 \frac{ \partial u^i}{ \partial x_j} \omega^j$.

 Letting $\psi(x,t)$ denote the vector stream function, given by $u =- \operatorname{curl} \psi$, we have that
 $\Delta \psi = \omega$, and hence  $\psi(x) = {\frac{1}{4\pi}} \int_{ \mathbb{R}^3  } \frac{\omega(y)}{|x-y|} dy$.  It follows that
 \begin{equation}\label{BS3d}
u(x,t) = \int_{ \mathbb{R}^2  } \K(x-y) \omega (y) dy \,,
\end{equation} 
where $\K(x) = {\frac{1}{4\pi}}  {\frac{x \times  \cdot }{|x|^3}}$ is the  Biot-Savart 3x3 matrix kernel.

What  type of vortex evolution in three space dimension is analogous to the 2-D vortex patch problem?  The answer is as follows: we suppose
that at time $t=0$,
 $\Omega^+(0)$ denotes an open bounded subset of $ \mathbb{R}^3$ which is 
diffeomorphic to the open unit ball $B=\{ x \in \mathbb{R}^3  \ : \ |x| < 1\}$.  We then let $\Gamma(0)=\partial\Omega^+(0)$,
and  define $\Omega^-(0) = \overline{\Omega^+(0)}^c$.
We choose an initial  divergence-free velocity field $u_0(x) = u^+_0 (x) {\bf 1}_{ \overline{\Omega ^+(0)}} + u_0 ^- (x) {\bf 1}_{ \Omega ^-(0)}$
 such that 
the initial vorticity vector $ \omega_0 = \operatorname{curl} u_0 \in L^ \infty ( \mathbb{R}^3  )$ and satisfies
\begin{subequations}
\label{vp3d-general}
\begin{align} 
\omega_0(x) & = \left\{
\begin{array}{ll}
\operatorname{curl} u_0^+(x)\,, & x \in \overline{\Omega^+(0)} \\
\operatorname{curl} u_0^-(x)\,, & x \in \Omega^-(0) 
\end{array}
\right. \,, \\
\jump{ \omega_0 \cdot n(\cdot , 0)} & =0 \,, \\
\jump{\omega_0 \times  n( \cdot , 0)} & \neq 0 \,, 
\end{align} 
\end{subequations} 
 where $ n(\cdot , 0)$ denotes the outward unit normal to
$\partial \Omega^+(0)$.   The velocity $u_0$  is continuous
across $\Gamma(0)$ while the tangential components of  $\omega_0$ are discontinuous.    The 3-D analogue of a 2-D vortex patch amounts to
choosing $u_0$ in such a way that $\curl u_0^- =0$ on $  \Omega^-(0)$ and hence, necessarily, $ \operatorname{curl} u_0^+ \cdot n(0) =0$
so that $ \omega_0$ is tangent to $\Gamma(0)$.

To explain this analogy, we first state the following existence theorem for the
Euler equations (\ref{euler}) with initial data $u(x,0) = u_0(x)$.    Gamblin \& Saint Raymond \cite{GaSa1995} proved that whenever
$\Gamma(0)$ is  $\mC^{1, \alpha }$, $ \alpha \in (0,1)$, $u_0 \in L^p( \mathbb{R}^3  )$, $1< p < \infty $, 
and $\omega_0 \in L^q( \mathbb{R}^3  )$, $1\le q < 3$ such that $\omega_0$ has $\mC^ \alpha $ regularity in directions tangent to 
$\Gamma(0)$, then there exists a unique solution $u \in L^ \infty (0,T; W^{1, \infty}( \mathbb{R}^3  )) \cap W^{1, \infty}(0,T; 
L^p( \mathbb{R}^3  ))$ to (\ref{euler}).   Furthermore, letting $ \eta(x,t)$ denote the Lagrangian flow of $u$, so that
\begin{subequations}
\label{lag-flow}
\begin{alignat}{2} 
\partial_t \eta (x,t) & = u(\eta(x,t),t) && \ \text{ for } \ t>0 \,, \\
\eta (x,0) & =x \,,&&
\end{alignat} 
\end{subequations}
and for each $t \in (0,T]$,  setting $\Gamma(t) = \eta( \Gamma(0), t)$, then $\Gamma(t)$ is a closed surface of class $\mC^{1, \alpha }$ and $\omega(t) \in 
L^q( \mathbb{R}^3  )$ such that  $\omega(t)$ has $\mC^ \alpha $ regularity in directions tangent to 
$\Gamma(t)$.

For each $t \in [0,T]$, the Lagrangian flow $\eta( \cdot ,t )$ is a  diffeomorphism with Jacobian determinant 
$\det \nabla \eta(x,t) =1$.   We set $\Omega^+(t) = \eta( \Omega^+(0), t)$ and $\Omega^-(t) = \eta( \Omega^-(0), t)$.
  Integrating the vorticity equation (\ref{vorticity3d}), we see that
\begin{equation}\label{lag_vor}
\omega( \eta(x,t),t) = \nabla \eta(x,t)  \cdot  \omega_0(x) \,,
\end{equation} 
where in components, $[\nabla \eta  \cdot  \omega_0]^i = \sum_{j=1}^3\frac{\partial \eta^i}{ \partial x_j} \omega_0^j$.   

We will set the 3-D vortex patch problem inside of a periodic box.
We let $\Omega$  denote  a periodic box $[- \ell ,\ell]^3$  in $ \mathbb{R}  ^3$ with opposite sides of the
box  identified with one another, and with $\ell$ taken sufficiently large so that $\overline{\Omega^+(0)} \subset \Omega$.
  Functions defined on $\Omega$ are $2\ell$-periodic in each of the three coordinate directions, i.e.,
$$
u(x + 2\ell e_i) = u(x) \ \ \forall x \in \mathbb{R}^3  , i=1,2,3 \,,
$$
were $e_1=(1,0,0)$, $e_2=(0,1,0)$ and $e_3=(0,0,1)$.

The {\it 3-D vortex patch problem}
has the following  initial data: 
\begin{subequations}
\label{vp3d}
\begin{align} 
\Gamma(0) & \text{ is a closed surface diffeomorphic to } \mathbb{S}  ^2 \,,\\
\Omega^+(0) & \text{ is an open set diffeomorphic to the unit ball in } \mathbb{R}^3  \,,\\
\Omega^-(0) & = \Omega - \overline{\Omega^+(0)}  \,,\\
u_0(x) & = {u_0 }_+(x) {\bf 1}_{ \overline{\Omega ^+(0)}} + {u_0}_- (x) {\bf 1}_{ \Omega ^-(0)}\,, \\
\operatorname{div} {u_0} &= 0  \,,\\
\omega_0 &= \operatorname{curl} u_0 \,, \\
\omega_0(x) & = \left\{
\begin{array}{ll}
\operatorname{curl} u_0^+(x)\,, & x \in \overline{\Omega^+(0)} \\
0\,, & x \in \Omega^-(0) 
\end{array}
\right. \,, \\
\operatorname{curl} u_0^+ \cdot n(\cdot , 0) & =0 \text{ on } \Gamma(0) \,, \\
\operatorname{curl} u_0^+ \times  n(\cdot , 0) & \neq 0 \text{ on } \Gamma(0) \,, \\
\int_\Omega u_0(x) dx &=0 \,.
\end{align} 
\end{subequations} 
 We then
call $\Omega^+(0)$ the initial {\it vortex patch} and $\Gamma(0)$ the initial vortex patch boundary.   The identity (\ref{lag_vor}) shows that
for each $t \in [0,T]$, $\omega( \cdot  ,t) = 0$ in $\Omega^-(t)$ and that $\omega( \cdot , t)\cdot n(\cdot , t) =0$ on $\Gamma(t)$.   In particular, 
if the initial vorticity is supported in a set which is diffeomorphic to a ball, then the vorticity stays supported in a  set diffeomorphic to a ball for
 all time $t \in [0,T]$ for which the solution exists.   In (\ref{vp3d}), we could instead set 
 $\Omega^-(0)  = \mathbb{R}   ^3 - \overline{\Omega^+(0)}$.

\subsection{Statement of the main result}
Because of the singular
nature of $\nabla K$, it is difficult to establish regularity for higher-order derivatives of $u$ with the formula (\ref{problematic})  .   
By taking a 
different approach, however, we shall prove that the velocity field indeed enjoys higher-order Sobolev regularity on both sides of the vortex
patch boundary.   In particular, for the 2-D vortex
patch problem defined in Section \ref{2dpatch}, we have the following
\begin{theorem}[Regularity of velocity field in 2-D]\label{thm1} Given initial data (\ref{vp2d}) and a global-in-time solution to the 2-D vortex patch problem satisfying
$$
\frac{1}{|z|_*(t)}+ \| z( \cdot , t)\|_{H^{k-0.5}( \mathbb{S}  ^1)} + \| \nabla u ( \cdot , t) \|_{ L^ \infty ( \mathbb{R}^2  )}   \le F(t)
$$
for  $t \in [0, \infty )$ and $k \ge 4$, the  velocity field satisfies $u^+( \cdot , t)  \in H^k(\Omega^+(t))$ and 
$u^-( \cdot , t)  \in H^k_{\loc}(\Omega^-(t))$, and
$$
\| u^+( \cdot , t) \|_{H^k(\Omega^+(t))} + \| u^-( \cdot , t) \|_{H^k(\Omega^-(t)) \cup B(0,R(t) ) )} \le G(t) \,,
$$
where $B(0,R(t))$ is a ball centered at $0$ with radius $R(t)>0$ such that $\Gamma(t) \subset B(0,R(t))$, and  $G(t)> 0$ is a function of $F(t)$,
defined in (\ref{bounds}),  with
$G(t) < \infty $ for any $t < \infty $.
\end{theorem} 

\begin{remark}
Notice that both velocity vector fields $u^+$ and $u^-$ gain a half-derivative of regularity with respect to the regularity of the vortex patch
boundary $\Gamma(t)$.   This is very natural in Sobolev spaces $H^k$, but requires us to locally extend our 1-D parameterization $z( \cdot , t)$
to  a 2-D local diffeomorphism $\theta^+( \cdot , t)$ and $\theta^-(\cdot , t)$ which also gains a half-derivative of regularity.    This is accomplished
by a specially chosen elliptic extension which we describe in Section \ref{sec::flatten}.   On the other hand, if we had assumed instead that the
parameterization $z( \cdot  , t) \in H^k( \mathbb{S}  ^1)$, then a standard local ``graph'' extension would have sufficed.  More 
specifically, if $z( \cdot , t)$ is given locally by the graph $(x_1, h(x_1))$, then $(x_1, x_2+ h(x_1))$ provides a local extension to a diffeomorphism, 
but does not gain a half-derivative of regularity.
\end{remark}

\begin{remark} Without any change to our proof, the  initial data (\ref{vp2d}) can be replaced by the more general initial data 
$$
\omega _0(x) =\left\{
\begin{array}{ll}
\omega_0^+ (x) \,, & x \in \overline{\Omega^+(0)} \\
\omega_0^-(x)\,, & x \in \Omega^-(0)
\end{array}
\right.
$$
for any functions $\omega^+_0 \in H^{k-1}(\Omega^+_0)$ and $\omega^-_0 \in H_{\loc}^{k-1}(\Omega^-_0)$, $k \ge 4$.  
\end{remark}

\begin{remark} In fact, Theorem \ref{thm1} is true for $k\ge 3$, but the proof requires one less regularization step for $k\ge 4$.
\end{remark}

%

Whereas 
Chemin \cite{Chemin1995} and Bertozzi \& Constantin \cite{BeCo1993} have established regularity of the contour $\Gamma(t)$ for the 2-D
vortex patch problem, the regularity of the 3-D vortex patch boundary $\Gamma(t)$ is limited to $\mC^{1, \alpha }$ in the analysis
of  Gamblin \& Saint Raymond \cite{GaSa1995}.    As our final result, we simultaneously establish  an existence theory in Sobolev spaces for
the 3-D vortex patch problem, as well as  the
Sobolev-class regularity of the 2-D closed surface  $\Gamma(t)$ and the velocity fields $u^+$ and $u^-$.

\begin{theorem}[Existence and regularity for the 3-D vortex patch boundary and velocity fields]\label{thm3} For $k \ge 3$, if $\Gamma(0)$ is a closed surface of Sobolev class $H^{k-0.5}$, and $u_0 \in H^1(\Omega)$ with $u_0^+ \in H^k( \Omega^+(0))$, $u_0^- \in H^k(\Omega^-(0))$ and
satisfying (\ref{vp3d}), then there is a time $T>0$ such that there exists a unique solution to the 3-D vortex patch problem, and for each
$t \in [0, T]$, the vortex patch boundary $\Gamma(t)$ is in $H^{k-0.5}$, 
$u^+( \cdot , t) \in H^k( \Omega^+(t))$,  and $u^-( \cdot , t) \in H^k( \Omega^-(t))$.
\end{theorem} 

\begin{remark}
The more general initial data (\ref{vp3d-general}) can replace (\ref{vp3d}) in Theorem \ref{thm3}.
\end{remark} 

\noindent
{\bf Notation.} We will denote the partial derivative $\frac{\partial f}{\partial x_j}$ by $f,_j$ for $j=1$, $2$, or $3$.   We will use the Einstein summation
convention, wherein repeated indices are summed from $1$ to $\n$, with $\n$ equaling either $2$ or $3$.

\subsection{Outline of the paper}
In Section \ref{sec2}, we define the strong form of the two-phase elliptic problem that the two-dimensional stream function must satisfy, and we also
define the associated variational formulation.   In Section \ref{sec::flatten}, we define the local diffeormorphisms  that we use to locally flatten the
vortex patch boundary; these diffeomorphisms gain one-half derivative of interior regularity in $H^k$ spaces relative to the regularity of the vortex
patch boundary.   Section \ref{sec4} is devoted to the Sobolev regularity theory of the fluid velocities $u^+( \cdot ,t)$ and $u^-( \cdot ,t)$ in the
2-D vortex patch problem.  

The 3-D vortex patch problem is studied in Section
\ref{sec5}.   After defining the two-phase-elliptic problem for the fluid velocity, we simultaneously
prove existence of solutions and establish the regularity theory for both the vortex patch boundary  $\Gamma(t)$ and the velocity fields $u^+( \cdot, t)$
and $u^-( \cdot , t)$; this is done in the Lagrangian framework. 
  Finally, in Section \ref{sec::elliptic_regularity}, we establish the fundamental  regularity estimates for the two-phase elliptic problem with Sobolev-class
coefficients in $\n$-dimensions (which arises in many applications, including the vortex patch problem).
   For completeness, we include a short appendix with some basic inequalities that are used in Section  \ref{sec::elliptic_regularity}.

\section{A two-phase elliptic problem for the 2-D stream function}\label{sec2}

The 2-D vortex patch problem has been previously studied using the evolution equation for the parameterization of the contour $z( \cdot ,t)$ 
(\cite{Chemin1993, BeCo1993}); see also \cite{CoTi1988} for perturbations of circular patches and \cite{Co2015} for elliptical patches).  We will take
a different approach.
 
 While not necessary, it is convenient to introduce the stream function formulation of the problem. 
  Let $\psi (x,t) = \psi^+ (x,t) {\bf 1}_{ \overline{\Omega ^+(t)} } + \psi ^- (x,t) {\bf 1}_{ \Omega ^-(t)}$.   
  We set $\jump{F} = F^+ - F^-$ on $\Gamma(t)$, and let $n( \cdot ,t)$ denote the outward unit normal to
 $\Gamma(t)$, and $\tau( \cdot , t)$ denote the unit tangent vector to $\Gamma(t)$.

 For each time $t \in [0, \infty )$,  the bounds (\ref{bounds}) show that
 $\nabla u( \cdot , t) \in L^{ \infty }( \mathbb{R}^2  )$; thus, $u( \cdot , t)  \in \hol$ and so the
 stream function $\psi( \cdot , t)  \in \htl$ is a  solution to the following
 two-phase elliptic problem for each fixed $t \in [0, \infty )$:
 \begin{subequations}\label{strong}
\begin{alignat}{2}
-\Delta \psi^+ ( \cdot ,t)&= -1 \qquad&&\text{in}\quad\Omega(t)^+\,, \\
\Delta \psi^- ( \cdot ,t)&= 0 \qquad&&\text{in}\quad\Omega(t)^-\,, \\
\jump{\psi (\cdot , t)} &= 0 \qquad&&\text{on}\quad\Gamma(t)\,,\\
\Bigjump{\frac{\partial \psi}{\partial n} ( \cdot ,t)} &= 0 \qquad&&\text{on}\quad\Gamma(t)\,.
\end{alignat}
\end{subequations}
The fact that $\psi( \cdot , t)  \in \htl$ means that the interface jump condition (\ref{strong}d) holds in $H^{0.5}(\Gamma(t))$.

  For each time $t \in [0, \infty )$,  (\ref{strong}) has the following weak formulation:
 \begin{equation}\label{weak}
 \int_{\Omega^+(t)} \nabla \psi ^+ ( \cdot, t) \cdot  \nabla \phi \, dx +  \int_{\Omega^-(t)} \nabla \psi ^-(\cdot , t)\cdot  \nabla \phi \, dx =
 - \int_{\Omega^+(t)} \phi\, dx \ \ 
 \forall \phi \in H^1 ( \mathbb{R}  ^2 ) \,.
 \end{equation} 
 From the bounds (\ref{bounds}), the stream-function satisfies
\begin{equation}\label{bounds2}
 \| \psi( \cdot , t) \|_{H^2(B(0,R(t))} \le F(t) \,,
\end{equation} 
where $B(0,R(t))$ is a ball centered at $0$ with radius $R(t)>0$ such that $\Gamma(t) \subset B(0,R(t))$.

%

\section{Locally flattening the boundary $\Gamma(t)$}\label{sec::flatten}
 We construct  local diffeomorphisms in small neighborhoods of $\Gamma(t)$ which  locally ``flatten'' the
vortex patch boundary, and which gain one-half derivative of regularity in the interior with respect to the regularity of the parameterization $z( \cdot ,t)$. There are other methods to construct regularizing diffeomorphisms (see, for example, \cite{La2005, CoSh2007, ShZe2008}), but the method we
present appears quite natural for arbitrary geometries.

Let $D^+=\{x \in \mathbb{R}^2  \ : \ |x| < 1\}$ denote the open unit ball in $ \mathbb{R}^2  $ with boundary $ \mathbb{S}  ^1 
=\{x \in \mathbb{R}^2  \ : \ |x| = 1\}$, the unit circle.   For each $t \in [0, \infty )$, we solve the following elliptic equation for 
$Z(r, \theta ,  t)$:
\begin{subequations}
\label{tz-equation}
\begin{alignat}{2} 
\Delta ^2 Z^+ & = 0 && \ \ \text{ in } D^+ \,, \\
Z^+& = z && \ \ \text{ on } \mathbb{S}  ^1 \,, \\
\frac{\partial Z^+}{ \partial r} & =  \frac{\partial z^\perp}{\partial \theta} && \ \ \text{ on } \mathbb{S}  ^1 \,.
\end{alignat}
\end{subequations}
The unique solution $Z^+(r, \theta, t)$ to (\ref{tz-equation}) satisfies the estimate
\begin{equation}\label{Z-estimate}
\| Z ^+( \cdot ,\cdot , t) \|_{H^k(D^+)} \le C \| z (\cdot , t) \|_{H^{k-0.5}( \mathbb{S}  ^1)} \,,
\end{equation} 
and we are considering integers $k \ge 4$.
The boundary conditions  (\ref{tz-equation}b,c) show that 
$$\det \nabla Z^+ (1, \theta,t) = |\partial_\theta z(\theta,t)|^2\,.$$
From the
definition (\ref{chordarc}) of $|z|_*(t)$  and its lower-bound given by (\ref{bounds}), it is proven in \cite{MaBe2002} that there exists a function 
$ \alpha (t) >0$ such that $ \alpha (t) \le \min_{\theta \in \mathbb{S}  ^1}  |\partial_\theta z(\theta,t)|^2$.   Hence,
$\det \nabla Z^+ (1, \theta,t) \ge \alpha (t) >0$.
This shows that $Z^+$ is locally injective around each point on $\mathbb{S}^1$.

Next, we define
$D^- = \{x \in \mathbb{R}^2  \ : \ 1 <  |x| < R(t) \}$, where $R(t)>0$ is  chosen sufficiently large so that the ball $B(0,R(t))$ contains 
$\Gamma(t)$.  
We let $Z^-(r, \theta , t)$ solve
\begin{subequations}
\label{tz-equation2}
\begin{alignat}{2} 
\Delta ^2 Z^- & = 0 && \ \ \text{ in } D^- \,, \\
Z^-& = z && \ \ \text{ on } \mathbb{S}  ^1 \,, \\
Z^-& = \id && \ \ \text{ on } \{r = R(t)\} \,, \\
\frac{\partial Z^-}{ \partial r} & =  \frac{\partial z^\perp}{\partial \theta} && \ \ \text{ on } \mathbb{S}  ^1 \,, \\
\frac{\partial Z^-}{ \partial r} & = e_r && \ \ \text{ on }  \{r = R(t)\} \,,
\end{alignat}
\end{subequations}
where $e_r$ denotes the unit basis vector $(\cos \theta, \sin \theta)$.
Again, we see that
the unique solution $Z^-(r, \theta, t)$ to (\ref{tz-equation2}) satisfies the estimate
\begin{equation}\label{Z-estimate2}
\| Z ^-( \cdot ,\cdot , t) \|_{H^k(D^-)} \le C \| z (\cdot , t) \|_{H^{k-0.5}( \mathbb{S}  ^1)} \,.
\end{equation} 
We define the map $Z = Z^+ {\bf 1}_{\overline{D^+} } + Z^- {\bf 1}_{D^-}$.  Due to the boundary conditions (\ref{tz-equation}b,c) and
(\ref{tz-equation2}b,d) and the Sobolev embedding theorem, the map $(r, \theta) \mapsto Z(r, \theta, t)$ is $\mC^1$, and for any 
point $\theta \in \mathbb{S}  ^1$, there
exists a ball $B(\theta, \varepsilon(t)) \subset \mathbb{R}^2  $, centered at $\theta$ with radius $ \varepsilon (t)>0$ taken sufficiently small,
such that $Z( \cdot , \cdot , t) $ is injective on $B(\theta, \varepsilon(t))$.

Next, we show that  for $ \epsilon >0$ sufficiently small,  the image  $Z^+( 1- \epsilon ,\theta , t)$ is contained in $\Omega^+(t)$, 
and similarly,
that  the image $Z^-( r,\theta , t)$ is contained in $\Omega^-(t)$.
To that end, let $\theta_0(t)$ denote the point in $ [0, 2\pi] $ at which the maximum value of $z( \theta , t) \cdot e_2$ occurs.
We assume that the tangent vector $\partial_\theta z( \theta_0(t), t) = \beta(t) \, e_1$ for some $ \beta(t) >0$ (for, 
 otherwise,  we can reverse the orientation of the parameterization).
Hence,  $\partial_\theta z^\perp ( \theta_0(t), t)=  \beta(t) \, e_2$. This shows that $\frac{\partial Z^+_2}{\partial x_2}(1,\theta_0(t),t)>0$, 
which in turn implies that $Z^+_2(1-\epsilon,\theta_0(t),t)<Z^+_2(1,\theta_0(t),t)$ which proves that, for $ \epsilon >0$ sufficiently small,
for all $r \in [1- \epsilon ,1)$ and $\theta \in [ \theta_0(t) - \epsilon , \theta_0(t) + \epsilon ]$,
$$Z^+(r, \theta ,t) \cdot e_2 <  z(\theta_0(t),t)\cdot e_2\,. $$
Therefore, $Z^+$ maps a local neighborhood of $\theta_0(t)$ (in $D^+$)  into 
$\Omega^+(t)$. 
Since $Z^+$ is locally injective around $\mathbb{S}_1$, this means that the image of any $Z^+(1-\epsilon,\cdot,t)$ 
(for $\epsilon>0$ small enough) stays in $\Omega^+(t)$, otherwise it would intersect $\Gamma(t)$, which we shall next  prove that
this  cannot occur.   Similarly, 
the image of any $Z^-(1+\epsilon,\cdot,t)$ stays in $\Omega^-(t)$.

We next prove that for $\epsilon>0$ sufficiently small, 
$$ Z^+(1-\epsilon,\theta,t) \cap \Gamma(t) = \emptyset \ \ \forall \theta \in \mathbb{S}  ^1 \,.$$
Since $\det \nabla Z^+ (1, \theta,t) \ge \alpha (t) >0$ for all $\theta \in \mathbb{S}  ^1$, 
by the inverse function theorem, there exists a small ball $B(\theta, \R(\theta)) \subset \mathbb{R}^2  $, centered at 
$\theta \in \mathbb{S}  ^1$ with radius $\R(\theta)>0$, such that 
$Z^+( \cdot , \cdot , t)$ is a $\mathcal{C}^1$-diffeomorphism between $ D^+ \cap B(\theta, \R(\theta))$ and 
$ Z^+( D^+ \cap B(\theta, \R(\theta)), t)$, as well as a 
 homeomorphism between $\overline{D^+\cap B(\theta ,\R(\theta))}$ and $Z^+(\overline{D^+\cap B(\theta ,\R(\theta))},t)$. 
 Since the compact set $\mathbb{S}^1$ is covered by $\bigcup_{ \theta \in \mathbb{S}  ^1} B(\theta ,\R(\theta))$, 
 we can extract a finite subcover  $ \bigcup_{i=1}^N B(\theta_i,\R(\theta_i))$, where $\theta_i$, $i=1,...,N$ are points in $ \mathbb{S}  ^1$.

Let $ \mathfrak{A}  ^ \epsilon = \{  x\in \mathbb{R}^2  \ : \ 1- \epsilon \le |x| < 1\}$ denote an annulus.  We choose $ \epsilon >0$  small enough so that $ \mathfrak{A}  ^ \epsilon \subset \bigcup_{i=1}^N B(\theta_i,\R(\theta_i))$.  With $(r, \theta) \in \mathcal{A} ^ \epsilon $ fixed, we
choose $i \in \{1,...,N\}$ 
such that $(r,\theta)\in B(\theta_i,\R(\theta_i))\cap D^+$. 
Since $Z^+( \cdot , \cdot , t) $ is $\mathcal{C}^1$ diffeomorphism between 
$D^+\cap B(\theta_i,\R(\theta_i))$ and $Z^+(D^+\cap B(\theta_i,\R(\theta_i)),t)$, then
 $Z^+(r,\theta,t)\in Z^+(D^+\cap B(\theta_i,\R(\theta_i)),t)$.   Furthermore, 
 as $Z^+$ is an homeomorphism between $\overline{D^+\cap B(\theta_i,\R(\theta_i))}$ and 
 $Z^+(\overline{D^+\cap B(\theta_i,\R(\theta_i))},t)$, then
  $Z^+(r,\theta,t) \not\in Z^+(\partial[{D^+\cap B(\theta_i,\R(\theta_i))}],t)$, which implies that $Z^+(r,\theta,t) \not\in 
  z([\theta_i-\R(\theta_i),\theta_i+\R(\theta_i)],t)\subset\Gamma(t)$. 

In summary, we have shown that for $(r,\theta)\in B(\theta_i,\R(\theta_i))\cap D^+$,
$$Z^+(r,\theta,t) \text{ is in the interior of }
Z^+(\overline{D^+\cap B(\theta_i,\R(\theta_i))},t)$$
with
$$\text{diameter}\left( Z^+(\overline{D^+\cap B(\theta_i,\R(\theta_i))},t) \right)  \le  2 \|\nabla Z^+\|_{L^\infty(D^+)} \R(\theta_i) \,.$$

From the positive lower bound  (\ref{bounds}) on the function  $|z|_*(t)$  in (\ref{chordarc}), there exists $\epsilon_0>0$ such that for any 
$x\in\Gamma(t)$, $B(x,\epsilon_0)\cap\Omega^+(t)$ does not contain any point of $\Gamma(t)$;  therefore, choosing the radius $\R(\theta)$
such that
$$2 \|\nabla Z^+\|_{L^\infty(D^+)} \R(\theta_i)< \epsilon_0\,,$$
(and increasing $N$ if necessary)
we have that $Z^+({D^+\cap B(\theta_i,\R(\theta_i))},t)$ does not contain any point of $\Gamma(t)$, which shows that 
$Z^+(r,\theta,t) \not\in \Gamma(t)$ as desired.   A similar argument shows that for $(r,\theta)\in B(\theta_i,\R(\theta_i))\cap D^-$,
$Z^-( r,\theta , t)$ is contained in $\Omega^-(t)$.


Thus, for each $\theta_i \in \mathbb{S}  ^1$, $i \in \{1,...N\}$, 
let $\U_i(t)= B(\theta_i ,  \R(\theta_i))\subset \mathbb{R}^2$, and let 
$\V_i(t) = Z(\U_i(t),t)$. 
The map $Z$ is then a $\mC^1$ diffeomorphism of $\U_i(t)$ onto  $\V_i(t)$, and 
due to the estimates (\ref{Z-estimate}) and (\ref{Z-estimate2}),
\begin{align*} 
Z^\pm( \cdot ,\cdot ,t) : D^\pm \cap \U_i(t) \to \Omega^\pm(t) \cap \V_i(t) \text{ is an $H^k$ diffeomorphism} \,.
\end{align*}

Next, we flatten the boundary of $\U_i(t) \cap \mathbb{S}  ^1$.  
For each $i \in \{1,...,N\}$, $\U_i(t) \cap \mathbb{S}  ^1$ is a graph given by $(x_1, h_i(x_1,t))$ where each $h_i( \cdot ,t)$ is
$\mC^ \infty $.   We define the $C^\infty$ local diffeomorphisms
$\vartheta_i^\pm(t)(x_1,x_2) = (x_1, x_2 \pm h_i(x_1,t))$ with $\det \nabla \vartheta_i^\pm(t) =1$, and we set 
$$B_\pm^i = [\vartheta_i^\pm(t)] ^{-1} ( \U_i(t) \cap D^\pm) \text{ and  }
B_0^i = [\vartheta_i^+(t)] ^{-1} ( \U_i(t) \cap \mathbb{S}  ^1) \,.$$
    The set $B_0^i \subset \{ x_2 =0\}$ is a flat boundary.


 Finally, we define
$\theta^\pm_i (t) = Z^\pm(t) \circ \vartheta^\pm_i(t)$.   Then
\begin{align} \label{charts}
\theta^\pm_i(t) : B^\pm \to \Omega^\pm \cap \V_i(t) \text{ is an $H^k$ diffeomorphism} \,,
\end{align} 
and thanks to (\ref{Z-estimate}), (\ref{Z-estimate2}), and (\ref{bounds}),  for each $i \in \{1,...,N\}$,
\begin{equation}\label{thetabound}
\Bigl\| \frac{1}{\det \nabla \theta^\pm_i (t)} \Bigr\|_{L^\infty  (B_\pm)} + \|\theta^\pm_i(t)\|_{H^k({B_\pm})}  \le \P( F(t)) \,,
\end{equation} 
where $\P(F(t))$ denotes a generic polynomial function of $F(t)$.
Furthermore,  if we set $\theta_i(t) = \theta_i^+(t){\bf 1}_{\overline{B_+^i}} + \theta_i^-(t){\bf 1}_{{B_-^i}}$, then
each $\theta_i(t) \in \mC^1(B)$, where $B = B_+ \cup B_-\cup B_0$.

\section{Regularity of the velocity field for 2-D vortex patches: 
Proof of Theorem \ref{thm1}}\label{sec4}
We first use the weak formulation (\ref{weak}) to build regularity of the stream function $ \psi^\pm $.  Interior regularity of $\Psi^\pm$
on sets
away from the patch boundary $\Gamma(t)$ is classical, so we focus our attention on regularity of $\Psi^\pm$ near $\Gamma(t)$.   We will
use the change-of-variables $\theta_i(t)$ given in (\ref{charts}).

\vspace{.1 in}
 \noindent
{\bf Step 1. The elliptic problem for $\Psi^\pm$ set on $B_\pm$.}
The weak formulation (\ref{weak}) can be written as
$$
 \int_{\V_i(t)\cap \Omega^+(t)} \nabla \psi ^+ ( \cdot, t) \cdot  \nabla \phi \, dx +  \int_{\V_i(t)\cap\Omega^-(t)} \nabla \psi ^-(\cdot , t)\cdot  \nabla \phi \, dx
  =
 - \int_{\V_i(t)\cap\Omega^+(t)} \phi\, dx 
 $$ 
 for all  test functions $\phi \in H^1_0 ( \V(t) )$ and each $i \in \{1,...,N\}$.
 
 With the collection of diffeomorphisms $\{\theta_i\}_{i=1}^N$ given in (\ref{charts}) for each $t\in [0, \infty )$, 
 we define 
$$A^\pm_i = [ \nabla \theta^\pm_i(t) ]^{-1}  \ \text{ and } \   J^\pm_i (t)= \det \nabla \theta_i^\pm(t)\,,$$
 and set 
 $$ \mathcal{A}^\pm _i =J^\pm_i [A^\pm_i ] [A^\pm_i ]^T\,.$$
It follows from  (\ref{thetabound}), and (\ref{bounds})  that for all $t \in [0, \infty) $, there exists 
a function $0 < \lambda_i(t)$ such that 
\begin{equation}\label{Abound}
w^T\, \mathcal{A}^\pm_i (x)\,  w  \ge \lambda_i(t) |w|^2 \ \forall w \in \mathbb{R}^2 \,, \ \ x \in B_\pm ^i \,.
\end{equation} 
To establish (\ref{Abound}), we drop the $i$ subscript (and superscript), and let  
$\tilde w_\pm =  J_\pm^{1/2} A_\pm\, w$. The left-hand side of (\ref{Abound}) is simply $ | \tilde w_\pm|^2$,  and  
$w= J_\pm^{-1/2} \nabla \theta^\pm
\tilde w_\pm$; therefore, 
$$\frac{|w|^2}{ \| J_\pm^{-1/2} \nabla \theta^\pm\|^2_{ L^ \infty (B^+) }} \le |\tilde w_\pm|^2 \,,$$
so that $\lambda (t) =  \| J^{-1/2}_\pm(t) \nabla \theta^\pm\|^{-2}_{ L^ \infty (B_\pm) }$, which has a strictly positive lower bound since 
$\lambda (t) ^{-1} =\| J_\pm^{-1/2} \nabla \theta^\pm\|^{2}_{ L^ \infty (B^+) }
\le \P(F(t))$ by (\ref{thetabound}).   Additionally, from (\ref{thetabound}), 
\begin{equation}\label{Abound2}
\| \mathcal{A}_\pm \|_{ H^{k-1}(B_\pm)} \le C \P(F(t)) \,.
\end{equation} 

   We set
 $$
 \Psi^\pm = \psi^\pm \circ \theta \,, \ \ \Phi = \phi\circ \theta \,.
 $$
Since $\phi \in H^1_0 ( \V(t) )$ and each $ \theta_i(t) \in \mC^1(B)$, it follows that $\Phi \in H^1_0(B)$, 
and can thus be used as a test function.
By another application of the change-of-variables formula, we  then have that
\begin{equation}\label{weak2}
 \int_{{B_+}} \mathcal{A}_+ ^{kj} \Psi ^+,_k  ( \cdot, t)\,  \Phi,_j \, dx + \int_{B_-} \mathcal{A}_- ^{kj} \Psi ^-,_k  ( \cdot, t)\,  \Phi,_j \, dx
  =
 - \int_{B_+} \Phi\, J_+ \,  dx  \ \ \forall  \Phi \in H^1_0 ( B ) \,.
 \end{equation} 
 
 \noindent
{\bf Step 2. $H^3$ regularity for $\psi^+$ and $\psi^-$.}  We set $k=4$ so  that $\theta^\pm \in H^4( B_\pm)$ and first establish that each 
$\psi^\pm$ is $H^3$.   We let $\{\zeta_i\}_{i=1}^N$ denote a smooth partition-of-unity, subordinate to the open cover $\U_i(t)$; in particular, 
$ 0 \le \zeta_i \le 1$ in $\mC^\infty_\cptspt(\U_i(t))$ denote a smooth cut-off function, $\sum_{i=1}^N \zeta_i =1$,  and let $ \xi_i= \zeta_i \circ \theta_i(t) $.  We define the horizontal convolution operator as follows: for $ \epsilon >0$ sufficiently small,
$$
\Lambda _ \epsilon F= \int_{ \bbR^{\n-1}} \rho_\epsilon (x_1-y_1) F(y_1,  x_2) dy_1 \,,
$$
where $\rho_\epsilon(x_1) = \epsilon^{-1} \rho( x_1/ \epsilon )$, and $\rho$ is the standard mollifier on $ \bbR$. 
We again drop the $i$ subscript, and substitute
$$
\Phi = \xi ^2 \Lambda _ \epsilon^2 \partial_1^4 (\xi^2\Psi)  \in H^1_0(B)\,, \ \ \Psi= {\bf 1}_{{\overline{B_+}}} \Psi^+ +{\bf 1}_{B_-} \Psi^-
$$
into (\ref{weak2}).
Since differentiation commutes with convolution, we have that
$$
\Phi,_j = \Lambda _ \epsilon ^2 \partial_1^4( \xi ^2\Psi),_j + 2 \xi \xi ,_j \Lambda _\epsilon ^2 \partial_1^4 ( \xi ^2 \Psi) \,.
$$
The variational formulation 
(\ref{weak2}) then takes the following form:
\begin{equation}
\label{170715.0}
 \I_1^\pm + \I_2^\pm
=-\int_{B_+} \partial_1^2 J_+\  \xi ^2 \Lambda _ \epsilon^2 \partial_1^4 (\xi^2\Psi) \ dx \,,
\end{equation}
where
\begin{align*}
\I_1^\pm & =  \int_{{B_\pm}} \Lambda_ \epsilon (  \xi ^2 \mathcal{A}_\pm ^{kj}  \Psi^\pm ,_k),_{11} \, \Lambda_\epsilon^2 (\xi^2  \Psi^\pm),_{j11} \, dx \,, \\
\I_2^\pm & = -2\int_{{B_\pm}}  (\xi \xi ,_j \mathcal{A}_\pm ^{kj} \Psi^\pm ,_k),_1 \, \Lambda_\epsilon^2 (\xi^2  \Psi^\pm) ,_{11} \, dx \,.
\end{align*}
Next, we see that
\begin{align*}
\I_1^\pm & = 
\underbrace{ \int_{{B_\pm}}  \mathcal{A}_\pm ^{kj} \Lambda_\epsilon ( \xi ^2 \Psi^\pm ),_{k11} \,  \Lambda_\epsilon ( \xi ^2\Psi^\pm ),_{j 11 } \, dx}
_{{ \I_1}_a^\pm}
+ \underbrace{ \int_{{B_\pm}} \big(\comm{\Lambda_\epsilon}{\mathcal{A}_\pm ^{kj}} ( \xi ^2\Psi^\pm) ,_{k11}\big) \, 
\Lambda_\epsilon ( \xi ^2\Psi^\pm) ,_{j 11 } \, dx}_{{ \I_1}_b^\pm} \\
& \qquad\qquad 
+ \underbrace{\int_{{B_\pm}} \Lambda _ \epsilon \left[ 2 {\mathcal{A}_\pm ^{kj}},_1 ( \xi ^2 \Psi),_{k1} + {\mathcal{A}_\pm ^{kj}},_{11} ( \xi ^2 \Psi),_k
-2( \xi\xi ,_k \mathcal{A}_\pm ^{kj} \Psi),_{11}\right]
\
\Lambda_\epsilon ( \xi ^2\Psi^\pm) ,_{j 11 } \, dx}_{{ \I_1}_c^\pm} \,,
\end{align*}
where
\begin{equation}\label{commutator}
\comm{\Lambda_\epsilon}{\mathcal{A}_\pm ^{kj} }( \xi ^2 \Psi^\pm) ,_{k11} = \Lambda_\epsilon ( \mathcal{A}_\pm ^{kj} ( \xi ^2 \Psi^\pm) ,_{k11} ) 
- \mathcal{A}_\pm ^{kj}\, \Lambda_\epsilon( \xi ^2 \Psi^\pm) ,_{k11}
\end{equation}
denotes the commutator of the horizontal convolution operator and multiplication by $ \mathcal{A}_\pm ^{kj}$. 
Using the lower-bound (\ref{Abound}), we see that
\begin{equation}
\label{170715.1}
\lambda(t) \|  \p_1^2 \Lambda _ \epsilon  \nabla( \xi ^2 \Psi^\pm) \|^2_{L^2({B_+})} \le
{ \I_1}_a^\pm  \,.
\end{equation}

We let $0 < \delta \ll 1$; we will make use of the Cauchy-Young inequality
$a\, b \le \delta \lambda(t) a^2 + \frac{1}{ 4 \delta \lambda(t)} b^2$ for $a,b\ge 0$.  
 
Using H\"{o}lder's inequality together with the
Sobolev  inequality  $\|f \|_{ L^p(B_\pm)} \le C \|f\|_{H^1(B_\pm)}$ for all $f \in H^1(B_\pm)$ and all $p\in[1, \infty )$,
we have that
\begin{equation*}
| { \I_1}_c^\pm   |\le C \|\mathcal{A}_\pm ^{kj},_{11}\|_{H^1 (B_\pm)} \| \Psi^\pm ,_k\|_{H^1(B_\pm)} \| \Lambda_\epsilon ( \xi ^2 \Psi^\pm) ,_{j 11 }\|_{L^2(B_\pm)}\,.
\end{equation*}
Thanks to (\ref{Abound2}) and (\ref{bounds2}), we then infer that
\begin{align*}
| { \I_1}_c^\pm  |&\le   \P(F(t)) \| \Lambda_\epsilon ( \xi ^2\Psi^\pm) ,_{j 11 }\|_{L^2(B_\pm)}\nonumber\\
&\le \delta\lambda(t) \| \Lambda_\epsilon( \xi ^2 \Psi^\pm) ,_{j 11 }\|^2_{L^2(B_\pm)} +\frac{\P(F(t))}{4\delta\lambda(t)}\nonumber\\
&\le \delta\lambda(t) \| \Lambda_\epsilon ( \xi ^2 \Psi^\pm) ,_{j 11 }\|^2_{L^2(B_\pm)} +(1+(\delta\lambda(t))^{-1}) \mathcal{P} (F(t))\,,
\end{align*}
where we continue to use $\P(F(t))$ to denote a generic polynomial function of $F(t)$.
A similar estimate can be established for the integral $\I_2^\pm$, which provides us with
\begin{align} 
| { \I_1}_c^\pm  |+|\mathcal{I} _2^\pm|
 \le [1+ ( \delta \lambda (t))^{-1} ] \mathcal{P}( F(t))  + \delta \lambda(t) \|\Lambda _ \epsilon ( \xi ^2\Psi^\pm),_{11} \|^2_{H^1( B_\pm)}  \,.
 \label{170715.2}
\end{align}
Also, the integral on the right-hand side of (\ref{170715.0})  has the same upper bound.    

It remains to establish such an upper bound for $| { \I_1}_b|$.  We set $g^\pm_k = ( \xi ^2 \Psi^\pm),_{k11}$; then, 
\begin{align}
 [\Lambda _ \epsilon (\mathcal{A}_\pm ^{kj} \,  g_k^\pm) - \mathcal{A}_\pm ^{kj}  \, \Lambda _ \epsilon  g_k^\pm](x_1,x_2)  
 = \int_{x_1 - \epsilon }^{x_1+\epsilon } \rho_ \epsilon ( { x_1-y_1})
 [ \mathcal{A}_\pm ^{kj} (y_1, x_2) -  \mathcal{A}_\pm ^{kj} (x_1,x_2) ]  g_k^\pm (y_1, x_2) dy_1 \,. \label{nnncs1}
\end{align}
From Morrey's inequality, for all $y_1 \in B(x_1, \epsilon )$,
\begin{equation}\label{nnncs2}
|\mathcal{A}_\pm ^{kj} (x_1,x_2) -  \mathcal{A}_\pm ^{kj} (y_1,x_2)| \le  C  \epsilon 
\sup_{ y_1 \in (x_1 - \epsilon, x_1+\epsilon) }  |{\mathcal{A}_\pm ^{kj}} ,_1 (y_1,x_2)|
\le C  \epsilon  \| \mathcal{A} \|_{H^{3}(B_\pm)}
\,.
\end{equation} 
Substituting (\ref{nnncs2}) into (\ref{nnncs1}) and using Young's inequality for convolution, together with
(\ref{Abound2}),  we see that
\begin{align*}
\|  \comm{\Lambda_\epsilon}{ \mathcal{A} ^{kj}_\pm} ( \xi ^2\Psi^\pm) ,_{k11} \|_{L^2(B_\pm)}
\le C { \epsilon } \| \mathcal{A} _\pm\|_{ H^{3}(B_\pm)} \| (\xi ^2\nabla \Psi^\pm) ,_{11} \|_{L^2(B_\pm)}
\le C { \epsilon }\P(F(t))\| \nabla (\xi ^2\Psi^\pm) ,_{11} \|_{L^2(B_\pm)} \,,
 \end{align*}
so that
\begin{align*} 
| { \I_1}_b| \le C { \epsilon }\P(F(t)) \| \nabla(\xi ^2 \Psi^\pm) ,_{11} \|_{L^2(B_\pm)}^2 \,.
\end{align*} 
We choose $ \epsilon $ sufficiently small so that $C { \epsilon }\P(F(t)) < \lambda(t)/2$.  
By choosing $ \delta >0 $ sufficiently small, we obtain from (\ref{170715.0}), (\ref{170715.1}) and (\ref{170715.2}) the estimate
\begin{equation}\nonumber
 \int_{{B_\pm}} |  \nabla  \Lambda_\epsilon( \xi^2 \Psi^\pm),_{11} |^2 \, dx \le [1+ \lambda (t)^{-1} ] \mathcal{P}( F(t)) \,.
\end{equation}
Passing to the limit as $ \epsilon \to 0 $, we find that
\begin{equation}\label{need1}
 \int_{{B_\pm}} \xi^2 | \nabla  \Psi^\pm,_{11} |^2 \, dx \le [1+ \lambda (t)^{-1} ] \mathcal{P}( F(t)) \,.
\end{equation}

From (\ref{strong}a,b), we have the following identity holding at any point of the interior of $B_\pm$:
\begin{align*} 
-\mathcal{A} ^{22}_\pm \Psi^\pm,_{221}  &= 2 \mathcal{A}^{21}_\pm \Psi^\pm,_{211} +  \mathcal{A}^{11}_\pm \Psi^\pm,_{111}+ 2 \mathcal{A}^{21}_\pm,_1 \Psi^\pm,_{21} + \mathcal{A}^{11}_\pm,_1 \Psi^\pm,_{11} \\
&+  \mathcal{A}^{22}_\pm,_1 \Psi^\pm,_{22} +   \mathcal{A}^{jk}_\pm,_{j1} \Psi^\pm,_{k} +   \mathcal{A}^{jk}_\pm,_{j} \Psi^\pm,_{k1}  \,.
\end{align*} 
The lower-bound (\ref{Abound}) shows that $ \mathcal{A} ^{22}_\pm \ge \lambda(t)$; hence
from (\ref{need1}), (\ref{Abound2}), and (\ref{bounds2}), 
\begin{equation}\label{need2}
 \int_{B_\pm}  \xi^2 |   \Psi^\pm,_{221} |^2 \, dx \le [\lambda (t) ^{-1} + \lambda (t)^{-2} ] \mathcal{P}( F(t)) \,.
\end{equation} 
Then, since
\begin{align*} 
-\mathcal{A} ^{22}_\pm \Psi^\pm,_{222}  &= 2 \mathcal{A}^{21}_\pm \Psi^\pm,_{221} +  \mathcal{A}^{11}_\pm \Psi^\pm,_{211}+ 2 \mathcal{A}^{21}_\pm,_2 \Psi^\pm,_{21} + \mathcal{A}^{11}_\pm,_2 \Psi^\pm,_{11} \\
&+  \mathcal{A}^{22}_\pm,_2 \Psi^\pm,_{22} +   \mathcal{A}^{jk}_\pm,_{j2} \Psi^\pm,_{k} +   \mathcal{A}^{jk}_\pm,_{j} \Psi^\pm,_{k2}  \,,
\end{align*} 
we use (\ref{need2}) to conclude that
\begin{equation}\label{need3}
 \int_{B_\pm}  \xi^2 |   \Psi^\pm,_{222} |^2 \, dx \le [\lambda (t) ^{-2} + \lambda (t)^{-3} ] \mathcal{P}( F(t)) \,.
\end{equation} 

Given the interior estimates, we
sum  (\ref{need1}),  (\ref{need2}), and (\ref{need3})  over our finite cover index $i=1,...,N$, and find that 
\begin{equation}\label{need6}
\| \psi^+ ( \cdot , t) \|_{H^3(\Omega^+(t)}^2 + \| \psi^- ( \cdot , t) \|_{H^3(\Omega^-(t) \cap B(0,R(t)))}^2 
 \le \mathcal{P}( F(t)) \,,
\end{equation} 
where we have used the fact that $ \lambda (t)^{-1} \le \P(F(t))$.
Then since $u^\pm = \nabla ^\perp \psi^\pm$,  (\ref{need6}) shows 
that
\begin{equation}\label{need7}
\| u^+ ( \cdot , t) \|_{H^2(\Omega^+(t)}^2 + \| u^- ( \cdot , t) \|_{H^2(\Omega^-(t) \cap B(0,R(t)))}^2 
 \le  \mathcal{P}( F(t)) \,.
\end{equation} 
Note that the estimate (\ref{need7}) has been obtained for the case that $\Gamma(t)$ is of Sobolev class $H^{3.5})$ so that we can indeed
build further regularity for $u^\pm$.

\vspace{.1 in}
\noindent
{\bf Step 3. $H^3$ regularity for $u^+$ and $u^-$.}   We will now use estimate (\ref{need7}) to build the $H^3$ regularity for $u^+$ and $u^-$.
On $\Gamma(t)$, 
we let $ \nabla _\tau u$ denote the directional derivative of $u$ in the direction $\tau$ and similarly, we let  $ \nabla _n u$ denote the directional derivative of $u$ in the direction $n$; for example, in components $\nabla _\tau u^i = u^i,_j \tau^j$.    We make use of the following identities on
$\Gamma(t)$:
 \begin{subequations}
 \label{divcurl}
\begin{align} 
\operatorname{div} u & = \nabla _ \tau u \cdot\tau +\nabla _n u \cdot n \,, \\
\operatorname{curl} u & = \nabla _ \tau u \cdot n - \nabla _n u \cdot \tau \,.
\end{align} 
 \end{subequations}
Since $u( \cdot ,t)$ is continuous across $\Gamma(t)$,  it follows that
$$
\jump{  \nabla _n u \cdot \tau }  = - \jump{ \operatorname{curl} u}  =  -1 \,.
$$
Then, using (\ref{divcurl}a), and the identity
$$
 \jump{  \nabla _n u  } = \jump{  \nabla _n u \cdot \tau } \tau + \jump{  \nabla _n u \cdot n } n \,,
$$
we see that
$$
 \jump{  \nabla _n u  } = - \tau \,.
$$

From (\ref{need7}),  the velocity field $u$ is a solution
 to the following two-phase elliptic problem:
\begin{alignat*}{2}
\Delta u^\pm &= 0 \qquad&&\text{in}\quad\Omega(t)^\pm\,, \\
\jump{u} &= 0 \qquad&&\text{on}\quad\Gamma(t)\,,\\
\bigjump{\nabla_n u} &= -\tau \qquad&&\text{on}\quad\Gamma(t)\,
\end{alignat*}
with variational form given by
\begin{equation}\label{weaku}
 \int_{\Omega^+(t)} \nabla u^+ ( \cdot, t): \nabla w \, dx +  \int_{\Omega^-(t)} \nabla u^-(\cdot, t) : \nabla w \, dx = \int_{\Gamma(t)} \tau( \cdot ,t) \cdot w\, dS(t) \ \ 
 \forall w \in H^1 ( \mathbb{R}  ^2; \mathbb{R}^2  ) \,,
 \end{equation} 
where $ A:B = A^i_j B^i_j$ for any 2x2 matrices $A$ and $B$.
  
Again dropping the subscript $i$,  we write (\ref{weaku}) locally as
\begin{equation}\label{weaku2}
 \int_{\V(t) \cap\Omega^+(t)} \nabla u^+ ( \cdot , t): \nabla w \, dx +  \int_{\V(t)\cap\Omega^-(t)} \nabla u^-(\cdot,  t) : \nabla w \, dx = 
 \int_{\V(t)\cap\Gamma(t)} \tau( \cdot ,t) \cdot w\, dS(t) 
\,.
 \end{equation} 
for all $w \in H^1_0 ( \V(t); \mathbb{R}^2  ) $.    We set $U= u \circ \theta $ and $W= w\circ \theta $.
 By the change-of-variables formula, (\ref{weaku2}) becomes
\begin{equation}\label{weaku3}
 \int_{{B_+}} \mathcal{A}_+ ^{kj} U_{,k} ^+ ( \cdot, t) \cdot   W_{,j} \, dx + \int_{B_-} \mathcal{A}_- ^{kj} U ^-_{,k}  ( \cdot, t) \cdot   W_{,j} \, dx
  =
 - \int_{B_0} \theta,_1\cdot W\, dS  \ \ \forall  W \in H^1_0 ( B; \mathbb{R}^2   ) \,.
 \end{equation} 
We then substitute
$$
W = \xi ^2 \Lambda _ \epsilon^2  \partial_1^4( \xi ^2   U ) \in H^1_0(B)\,, 
$$
into (\ref{weaku3}).  By repeating the identical argument of Step 2 above, we find that
\begin{equation}\label{need8}
\| u^+ ( \cdot , t) \|_{H^3(\Omega^+(t)}^2 + \| u^- ( \cdot , t) \|_{H^3(\Omega^-(t) \cap B(0,R(t)))}^2 
 \le \P(F(t) )\,.
\end{equation} 

\noindent
{\bf Step 4. $H^4$ regularity for $u^+$ and $u^-$.}   We continue to assume that $k=4$ so that the boundary $\Gamma(t)$ is of Sobolev
class $H^{3.5}$ and our change-of-variables $\theta_i^\pm(t) \in H^4(B_\pm)$.  We  will now show that $u^+$ and $u^-$ have 
$H^4$ regularity.

To do so, we 
let the test function   $W = -\xi ^2 \Lambda _ \epsilon^2  \partial_1^6( \xi ^2   U )$
in (\ref{weaku3}).  By a slight modification of Step 3, we find that
\begin{equation}\label{need9}
\| u^+ ( \cdot , t) \|_{H^4(\Omega^+(t)}^2 + \| u^- ( \cdot , t) \|_{H^4(\Omega^-(t) \cap B(0,R(t)))}^2 
 \le \P(F(t) )\,.
\end{equation}
There are new types of integrals that arise in establishing the $H^4$ regularity; namely, integrals that have highest-order derivatives on
both $U^\pm$ and $\theta^\pm$.

One of these integrals is analogous to one of the integrals in ${ \I_1}_c^\pm$ defined in Step 1 and is written as
\begin{equation*}
\mathcal{J} ^\pm  = \int_{{B_\pm}} \Lambda_\epsilon \left[  \mathcal{A}_\pm ^{kj},_{111} ( \xi ^2 U^\pm) ,_k\right] \,  \Lambda_\epsilon ( \xi ^2U^\pm) ,_{j 111 }\, dx\,.
\end{equation*}
We estimate the integral   $| \mathcal{J} ^\pm|$ using an  $L^2$-$L^\infty$-$L^2$ H\"{o}lder's inequality:
\begin{equation*}
|\mathcal{J}^\pm | \le \|\mathcal{A}_\pm ^{kj},_{111}\|_{L^2(B_\pm)} \| ( \xi ^2U)^\pm ,_k \|_{L^\infty(B_\pm)}\, \| \Lambda_\epsilon( \xi ^2 U)^\pm ,_{j 111 }\|_{L^2(B_\pm)}\,,
\end{equation*}
which, with the Sobolev embedding of $H^2(B_\pm)$ into $L^\infty(B_\pm)$, shows that
\begin{equation*}
|\mathcal{J} ^\pm |  \le C \|\mathcal{A}_\pm ^{kj},_{111}\|_{L^2(B_\pm)} \|U^\pm ,_k \|_{H^2(B_\pm)}\, \| \Lambda_\epsilon( \xi ^2 U)^\pm ,_{j 111 }\|_{L^2(B_\pm)}\,.
\end{equation*}
Using the estimate (\ref{thetabound}) with $k=4$ together 
with the previous lower-order estimate (\ref{need8})of $u^\pm$ in $H^3$, we obtain that
\begin{equation*}
|\mathcal{J} ^\pm |  \le \P(F(t)) \| \Lambda_\epsilon( \xi^2 U^\pm) ,_{j 111 }\|_{L^2(B_\pm)}\,,
\end{equation*}
which is just a linear term in $ \| \Lambda_\epsilon(\xi^2 U^\pm ),_{j 111 }\|_{L^2(B_\pm)}$, easily controlled by the {\it energy integral}
\begin{equation*}
\mathcal{I}_{1_a,4}^\pm= \lambda (t) \int_{{B_\pm}} |  \Lambda_\epsilon  \nabla(\xi^2 U)^\pm ,_{111} |^2\,  dx\,,
\end{equation*}
analogous to the term  $ {\mathcal{I} _1}_a^\pm$ in Step 2 above.

Other integral terms set on $B_{\pm}$ of this type arise and can be treated similarly. There is one slight variation:  the boundary integral term
\begin{equation*}
\mathcal{I}_\p=\int_{B_0} \partial_1^{3} \Lambda _ \epsilon ( \xi ^2\theta,_1) \cdot   \partial_1^{3}\Lambda _ \epsilon  (\xi ^2U) \ dS\,,
\end{equation*}
for which we simply notice that
\begin{align*}
|\mathcal{I}_\p|&\le  \|\partial_1^3 \Lambda _ \epsilon ( \xi ^2\theta,_1)\|_{H^{-0.5}(B_0)} \| \partial_1^{3}\Lambda _ \epsilon( \xi ^2  U)\|_{H^{0.5}(B_0)}\\
&\le C \|\theta,_1\|_{H^{3.5}(B_0)} \| \partial_1^{3}\Lambda _ \epsilon  ( \xi ^2U)\|_{H^{1}(B_+)}\,,
\end{align*}
where we have used the properties of the convolution operator for the first norm on the right-hand side, and the trace theorem for the second norm. This then provides us with
\begin{equation*}
|\mathcal{I} _{\p}|  \le \P(F(t)) \| \Lambda_\epsilon ( \xi ^2U^\pm ),_{111 }\|_{H^1 (B_+)}\,,
\end{equation*}
which is a linear term controlled in a similar manner as $\mathcal{I}_{1_a,4}^\pm$.

\vspace{.1 in}
\noindent
{\bf Step 5. $H^k$ regularity for $u^+$ and $u^-$.} Letting $W = (-1)^{k-1}\xi ^2 \Lambda _ \epsilon^2  \partial_1^{2(k-1)}( \xi ^2   U )$
in (\ref{weaku3}) and repeating Step 3, concludes the proof.
\qed

\section{Existence and regularity of the 3-D vortex patch boundary $\Gamma(t)$ and $u_\pm$: Proof of Theorem \ref{thm3}}\label{sec5}

\subsection{The two-phase elliptic problem for velocity}
As defined in Section \ref{3dpatch}, the vortex patch boundary $\Gamma(t)$ is a closed 2-D surface which is diffeomorphic to 
the unit sphere $ \mathbb{S}  ^2$, and that $\Omega^+(t)$ is an open subset of $ \mathbb{R}^3  $ such that $\partial \Omega^+(t) = \Gamma(t)$, and $\Omega^-(t)  = \mathbb{T}  ^3 - \overline{\Omega^+(0)}$.   We denote $\mathbb{T}  ^3$ by $\Omega$ in what follows,
and we set $\Omega^\pm = \Omega^\pm(0)$.

 We let $\tau_1( \cdot,  t)$ and $\tau_2(\cdot , t)$ denote an orthonormal basis of the tangent plane to 
each point of $\Gamma(t)$, so that $(\tau_1,\tau_2,n)$ is a direct orthonormal frame of $ \mathbb{R}^3  $.
We let $ \nabla _{\tau_\alpha} u$ ($\alpha=1,2$) denote the directional derivative of $u$ in the direction $\tau_\alpha$ and similarly, we let  
$ \nabla _n u$ denote the directional derivative of $u$ in the direction $n$; for example, in components $\nabla _{\tau_\alpha} u^i = u^i,_j (\tau_\alpha)^j$.    We make use of the following identities on
$\Gamma(t)$:
 \begin{subequations}
 \label{divcurl3d}
\begin{align} 
\operatorname{div} u & = \nabla _ {\tau_\alpha} u \cdot\tau_\alpha +\nabla _n u \cdot n \,, \\
\operatorname{curl} u & = (\nabla _ {\tau_2} u \cdot n - \nabla _n u \cdot \tau_2)\tau_1- (\nabla _ {\tau_1} u \cdot n - \nabla _n u \cdot \tau_1)\tau_2\,,
\end{align} 
 \end{subequations}
where we have used the fact that $ \operatorname{curl} u^+ \cdot n =0$ n $\Gamma(t)$ by (\ref{vp3d}h).
Since $u( \cdot ,t)$ is continuous across $\Gamma(t)$,  it follows that
\begin{alignat}{2} 
\jump{  \nabla _n u \cdot \tau_1 }  & = \jump{ \operatorname{curl} u\cdot\tau_2}  &&=  \operatorname{curl} u^+\cdot\tau_2\,, \\
-\jump{  \nabla _n u \cdot \tau_2 }  &= \jump{ \operatorname{curl} u\cdot\tau_1}  && =  \operatorname{curl} u^+\cdot\tau_1\,.
\end{alignat} 
Then, using (\ref{divcurl3d}a), and the identity
\begin{equation} \label{cs1.10.12}
 \jump{  \nabla _n u  } = \jump{  \nabla _n u \cdot {\tau_\alpha} } \tau_\alpha + \jump{  \nabla _n u \cdot n } n \,,
\end{equation} 
we see that
\begin{equation}\label{cs1.10.11}
 \jump{  \nabla _n u  } =  \operatorname{curl} u^+\cdot\tau_2\ \tau_1 - \operatorname{curl} u^+\cdot\tau_1\ \tau_2 \,.
\end{equation} 
The velocity field $u = u_+{\bf 1}_{\overline{\Omega^+(t)}} + u_-{\bf 1}_{\Omega^-(t)}$ is a weak solution
 to the following two-phase elliptic problem:
\begin{subequations}\label{strongu3d}
\begin{alignat}{2}
-\Delta u^+ &=  \operatorname{curl} \operatorname{curl} u^+ \qquad&&\text{in}\quad\Omega(t)^+\,, \\
\Delta u^- &=  0 \qquad&&\text{in}\quad\Omega(t)^-\,, \\
\jump{u} &= 0 \qquad&&\text{on}\quad\Gamma(t)\,,\\
\bigjump{\nabla_n u} &= \operatorname{curl} u^+\cdot\tau_2\ \tau_1 - \operatorname{curl} u^+\cdot\tau_1\ \tau_2  \qquad&&\text{on}\quad\Gamma(t)\,,
\end{alignat}
\end{subequations}
with variational (or weak) form given as follows:  For all  vector test-functions $w \in H^1 ( \Omega )$ given by
\begin{align} 
\label{weaku3d}
& \int_{\Omega^+(t)} \nabla u^+ ( \cdot, t): \nabla w \, dx +  \int_{\Omega^-(t)} \nabla u^-(\cdot, t) : \nabla w \, dx =
 \int_{\Omega^+(t)} \operatorname{curl}  u^+ ( \cdot, t) \cdot  \operatorname{curl}  w\, dx  \nonumber\\
 &\qquad \qquad \qquad \qquad + \int_{\Gamma(t)} [n\times \operatorname{curl} u^+]  \cdot w\, dS(t)\nonumber\\
&\qquad \qquad \qquad \qquad + \int_{\Gamma(t)} [\operatorname{curl} u^+( \cdot ,t)\cdot\tau_2\ \tau_1 - \operatorname{curl} u^+( \cdot ,t)\cdot\tau_1\ \tau_2]  \cdot w\, dS(t)  \,,
\end{align} 
where $ A:B = A^i_j B^i_j$ for any 3x3 matrices $A$ and $B$.
Next we notice that 
\begin{equation} \label{cs2.10.11}
n\times\operatorname{curl} u^+=n\times [\operatorname{curl} u^+\cdot\tau_1\ \tau_1+\operatorname{curl} u^+\cdot\tau_2\ \tau_2]
=\operatorname{curl} u^+\cdot\tau_1\ \tau_2-\operatorname{curl} u^+\cdot\tau_2\ \tau_1\,,
\end{equation} 
so that the boundary integral terms of (\ref{weaku3d}
) cancel each other, and we are left with
\begin{equation} 
\label{weaku3d2}
 \int_{\Omega^+(t)} \nabla u^+ ( \cdot, t): \nabla w \, dx +  \int_{\Omega^-(t)} \nabla u^-(\cdot, t) : \nabla w \, dx =
 \int_{\Omega^+(t)} \operatorname{curl}  u^+ ( \cdot, t) \cdot  \operatorname{curl}  w\, dx  \ \ \forall w \in H^1(\Omega)   \,.
\end{equation} 

We let $\eta(x,t)$ denote the Lagrangian flow of $u$, as defined in (\ref{lag-flow}).
We set $v^\pm = u^\pm \circ \eta$ and
we define $A(x,t) = [ \nabla \eta(x,t)] ^{-1}$.   Then,  letting $\phi = w \circ \eta$, (\ref{weaku3d2}) can be written as
\begin{equation} 
\label{weaku3dlag}
 \int_{\Omega^+} \mathcal{A} ^{jk} \frac{\partial v^+}{\partial x_j} \cdot  \frac{\partial \phi}{\partial x_k}  \, dx 
 +  \int_{\Omega^-} \mathcal{A} ^{jk} \frac{\partial v^-}{\partial x_j} \cdot  \frac{\partial \phi}{\partial x_k} \, dx =
 \int_{\Omega^+}  \left[\operatorname{curl}  u^+\right]  \circ \eta  \cdot  \left[ \operatorname{curl}(  \phi \circ \eta ^{-1} ) \right] \circ \eta \, \det \nabla \eta dx   
\end{equation} 
for all $ \phi \in H^1(\Omega) $, 
where
$$
\mathcal{A} ^{jk} = A^j_i A^k_i  \det \nabla \eta\,.
$$
For solutions to the Euler equations (\ref{euler}), $ \operatorname{div} u=0$ so that $\det \nabla \eta=1$, but the general form  
(\ref{weaku3dlag}) will be necessary for our fixed-point scheme.

\subsection{The fixed-point procedure for existence of solutions to the vortex patch problem}

In Section \ref{sec::elliptic_regularity}, we will establish 
the fundamental elliptic regularity results for a Lagrangian variational formulation  as in (\ref{weaku3d2}).
Using that regularity theory, we now prove the existence and regularity  of solutions to the 3-D vortex patch problem; our solutions have
smooth Sobolev regularity on both sides of the vortex patch boundary $\Gamma(t)$  and are globally in $H^1(\Omega)$.

\subsubsection{The functional framework}
We remind the reader that we use $\Omega$ to denote  a periodic box $[- \ell ,\ell]^3$  in $ \mathbb{R}  ^3$ with opposite sides of the
box  identified with one another, and with $\ell$ taken sufficiently large so that $\overline{\Omega^+(0)} \subset \Omega$.
  Functions defined on $\Omega$ are $2\ell$-periodic in each of the three coordinate directions, i.e.,
$$
u(x + 2\ell e_i) = u(x) \ \ \forall x \in \mathbb{R}^3  , i=1,2,3 \,,
$$
were $e_1=(1,0,0)$, $e_2=(0,1,0)$ and $e_3=(0,0,1)$.  
  Functions in
$ H^1(\Omega) $ satisfy periodic boundary conditions, and $H^1(\Omega)$ can be identified with $H^1 ( \mathbb{T}  ^3)$.

Given $T>0$ and $M>0$ assumed fixed, we work in the Lagrangian framework and define the bounded closed convex and nonempy set
\begin{equation}
\label{fp1}
{\bf V}_M^k=\{v\in L^2(0,T; H^1(\Omega))\cap L^2(0,T;H^k(\Omega^\pm));\ 
\  \|v\|_{L^2(0,T; H^1(\Omega))}+\|v^\pm\|_{L^2(0,T; H^k(\Omega^\pm))}\le M\}\,
\end{equation}
for integers $k\ge 3$.
For any $v\in {\bf V}_M^k$, we define the Lagrangian flow
\begin{equation}
\label{fp2}
\eta(x,t)=x+\int_0^t v(x,s)\ ds\,,
\end{equation}
which therefore, from (\ref{fp1}), satisfies $\eta\in \mC^0(0,T; H^1(\Omega))\cap \mC^0(0,T;H^k(\Omega^\pm))$.   Note, also, that since
the vortex patch boundary is transported by the fluid velocity, we have that 
$$
\Gamma(t) = \eta(\Gamma, t)\,.
$$
Hence, the regularity of the velocity field in $\Omega^+$ provides us with the regularity of $\eta$ in $\Omega^+$; the trace theorem then provides
the regularity of $\eta$ on $\Gamma$, and this in turn provides the regularity of the vortex patch boundary $\Gamma(t)$.

Since $\Omega$ is a periodic box, and hence convex, 
any two distinct points  $x$ and $y$ in $\overline{\Omega}$ can be connected by the straight-line segment $(x,y)$; therefore, 
by splitting the segment $(x,y)$ into  a finite union of subsegments $(x_i,x_{i+1})$, we can assume that each subsegment $(x_i,x_{i+1})$ is contained in 
either
$\Omega^+$ or $\Omega^-$.   It follows from (\ref{fp2}) that 
\begin{align*}
\eta(x,t)-\eta(y,t)=& x-y+\int_0^t v(x,s)-v(y,s)\ ds\\
=& x-y+\int_0^t v(x_1,s)-v(x_K,s)\ ds\\
=& x-y+\sum_{i=1}^{K-1} \int_0^t v(x_i,s)-v(x_{i+1},s)\ ds\,,
\end{align*}
which therefore shows by the fundamental theorem of calculus, that since each $(x_i,x_{i+1})$ is either contained in 
$\Omega^+$ or $\Omega^-$, that
\begin{equation*}
\left|\eta(x,t)-\eta(y,t)- x-y\right|\le  C\sum_{i=1}^{n-1} |x_i-x_{i+1}| \int_0^t \|\nabla v(\cdot,s)\|_{L^\infty(\Omega^+)}+\|\nabla v(\cdot,s)\|_{L^\infty(\Omega^-)}\ ds\,,
\end{equation*}
and from the Sobolev embedding theorem,
\begin{equation*}
\left|\eta(x,t)-\eta(y,t)- x-y\right|\le  C \sum_{i=1}^{n-1} |x_i-x_{i+1}| \int_0^t \|\nabla v(\cdot,s)\|_{H^2(\Omega^+)}+\|\nabla v(\cdot,s)\|_{H^2(\Omega^-)}\ ds\,.
\end{equation*}
From the definitions (\ref{fp1}) and (\ref{fp2}), it follows that
\begin{equation}
\label{fp4}
\left|\eta(x,t)-\eta(y,t)- x-y\right|\le  \sum_{i=1}^{n-1} |x_i-x_{i+1}| 2\sqrt{t} M \le 2\sqrt{T} M C |x-y|\,.
\end{equation}
We now choose $T$ such that
\begin{equation}
\label{fp5}
0< T\le\frac{1}{16M^2 C^2}\,,
\end{equation}
so that for any $x$ and $y$ in $\overline{\Omega}$,
\begin{equation*}
\left|\eta(x,t)-\eta(y,t)\right|\ge  \frac{1}{2} |x-y|\,,
\end{equation*}
which establishes the injectivity of $\eta$ in $\overline{\Omega}$. Furthermore, since
\begin{equation}
\label{fp6.0}
\left|\nabla\eta(x,s)-\text{Id}\right|\le \left|\int_0^t \nabla v(x,s)\ ds\right| 
\le \left|\int_0^t \|\nabla v^+(\cdot,s)\|_{L^\infty(\Omega^+)}+\|\nabla v^-(\cdot,s)\|_{L^\infty(\Omega^-)} \ ds\right|\le 2 C\sqrt{T} M\,,
\end{equation}
due to the continuity of the determinant at $\text{Id}$ in $\mathbb{R}^9$,  we can choose $T>0$ small enough,
 so that for all $x\in\Omega$ and $0\le t\le T$,
\begin{equation}
\label{fp6}
{\frac{3}{2}} \ge \text{det}\nabla\eta(x,t)\ge\frac{1}{2}\,,
\end{equation}
which shows, with the previously established injectivity,  that $\eta( \cdot ,t)$ is an $H^4$ diffeomorphism from $\Omega^\pm$ onto
the image $\eta( \Omega^\pm,t)$, and a homeomorphism from $\Omega$ onto $\eta(\Omega,t)$.  
Finally, by choosing $T$ sufficiently small we can ensure the strict positivity of the coeffiicient matrix $ \mathcal{A} $:
 for all $t\in[0,T]$,
\begin{equation}\label{Acoer}
w^T\, \mathcal{A}^\pm_i (x,t)\,  w  \ge {\frac{1}{4}}  |w|^2 \ \forall w \in \mathbb{R}^2 \,, \ \ x \in \Omega\,.
\end{equation} 

\subsubsection{The fixed-point procedure}
We define the Lagrangian curl operator $ \operatorname{curl} _\eta$ as follows: if $u(y,t)$ is an Eulerian vector, and $v = u \circ \eta$, then
we $\operatorname{curl} _\eta v = [ \operatorname{curl} u] \circ \eta $ where for any differential vector field $F$,  and for $i=1,2,3$,
\begin{equation}\label{curleta}
 [\operatorname{curl} _\eta F ]_i =\varepsilon_{i jk} \frac{\partial F^k}
{\partial x_r} A^r_j \,,
\end{equation} 
where $ \varepsilon_{ ijk} $ denotes the permutation symbol, so that $ \varepsilon_{ ijk} =1$ for even permutations, $ \varepsilon_{ ijk} =-1$ for odd
permutations, and $ \varepsilon_{ ijk} =0$ otherwise.
We will employ a fixed-point procedure on the variational equation (\ref{weaku3dlag}), which we write as
\begin{equation} 
\label{weaku3dlag2}
 \int_{\Omega^+} \mathcal{A} ^{jk} \frac{\partial v^+}{\partial x_j} \cdot  \frac{\partial \phi}{\partial x_k}  \, dx 
 +  \int_{\Omega^-} \mathcal{A} ^{jk} \frac{\partial v^-}{\partial x_j} \cdot  \frac{\partial \phi}{\partial x_k} \, dx =
 \int_{\Omega^+}  \left[\operatorname{curl}  u^+\right]  \circ \eta  \cdot  \operatorname{curl} _\eta \phi \, \det \nabla \eta dx   
\end{equation} 
for all $ \phi \in H^1(\Omega) $.   From (\ref{lag_vor}), 
\begin{equation}\label{cs10.08.1}
 \operatorname{curl} u \circ \eta =  \nabla  \eta    \cdot \omega_0 \,, \ \ \omega_0 = \operatorname{curl} u_0^+ {\bf 1}_{\Omega^+} \,.
\end{equation} 
Since $ \operatorname{div} \omega _0=0$, using the formula (\ref{cs10.08.1}), we see that
\begin{equation}\label{ncs10}
\int_{\Omega^+}\operatorname{curl} u^+ \circ \eta \ dx = \int_{\Gamma} \eta \  (\operatorname{curl} u_0^+ \cdot n( \cdot , 0)) \, dS(0)=0\,,
\end{equation} 
where the last equality follows from (\ref{vp3d}h).
 
Now, given $v$ in our convex set ${\bf V}_M^k$ and letting $\eta$ denote the homeomorphism defined in (\ref{fp2}), 
we define 
\begin{equation}
\label{fp8}
\mathcal{C}(v)(x,t) =  \nabla \eta(x,t)\cdot \omega_0(x) \ \  \text{ in } \ \Omega \,.
\end{equation}

Notice that
for any $x\in\Gamma$,  the trace on $\Gamma$ of $\mathcal{C}(v)(x,t)\cdot n(\eta(x,t),t)$ (the trace taken from from $\Omega^+$)
is zero, and is thus equal to the trace of of $\mathcal{C}(v)(x,t)\cdot n(\eta(x,t),t)$ evaluated from $\Omega^-$.   
 To see this, we use an important
geometric property of the inverse deformation matrix  $A(x,t) =  [ \nabla \eta(x,t)] ^{-1}$; namely,
if $N(x):= n(0,x)$ denotes the outward unit normal to $\p \Omega^+$ and if $n(\eta(x,t),t)$ denotes the outward unit normal to $\p \Omega^+(t)$, then
$$
n_i(\eta(x,t), t) = \frac{A^k_i N_k}{ |A^T N|} \,.
$$
Hence, it follows that
\begin{align}
\mathcal{C}(v)\cdot n \circ \eta =& \mathcal{C}(v)^i \frac{A^k_i N_k}{|A^T N|} =  \frac{1}{|A^T N|}  N_k A^k_i \frac{\p \eta^i}{\p x_l} \omega_0^l
=  \frac{1}{|A^T N|}  N_k \omega_0^k =0 \,,
\label{transport1.b}
\end{align}
where we have again used (\ref{vp3d}h) for the last equality.

Furthermore,  the same computation as in (\ref{ncs10}) shows that
\begin{align}
\int_{\Omega^+} \mathcal{C} (v) \ dx = \int_{\Gamma} \eta \  (\operatorname{curl} u_0^+ \cdot N) \, dS(0)=0\,.
\label{transport2}
\end{align}

Now, for each time $t\in [0,T]$, we construct a solution
$\bar{v}(\cdot,t)$ to the following variational problem:  
\begin{equation}
\label{fp9}
 \int_{\Omega^+} \mathcal{A} ^{jk} \frac{\partial v^+}{\partial x_j} \cdot  \frac{\partial \phi}{\partial x_k}  \, dx 
 +  \int_{\Omega^-} \mathcal{A} ^{jk} \frac{\partial v^-}{\partial x_j} \cdot  \frac{\partial \phi}{\partial x_k} \, dx 
 = \int_{\Omega^+} \mathcal{C}(v)\cdot \operatorname{curl} _\eta \phi \  \det \nabla \eta dx \ \ \forall \phi\in H^1(\Omega)\,.
\end{equation}
From (\ref{Acoer}) and
 the Lax-Milgram theorem, there exists a unique 
 periodic solution $\bar v( \cdot , t) \in H^1(\Omega)$ for each fixed $t\in[0,T]$, satisfying
\begin{equation}
\label{fp9.2}
\int_{\Omega}\bar{v}\ dx=0\,,
\end{equation}
Furthermore, since $\mathcal{C}(v) \in H^{k}\Omega^+$, $k\ge 2$,  we may integration-by-parts on the right-hand side of (\ref{fp9}).  We use
the fact that the cofactor matrix $a(x,t)$,  defined by $a= \det \nabla \eta \, A$, satisfies the Piola identity $\frac{\p}{\p x_k }a^k_i=0$ for $i=1,2,3$.
Thus, we see that (\ref{fp9}) can be written as follows: 
for all $\phi \in H^1(\Omega) $, 
\begin{align*}
 \int_{\Omega^+} \mathcal{A} ^{jk} \frac{\partial v^+}{\partial x_j} \cdot  \frac{\partial \phi}{\partial x_k}  \, dx 
 +  \int_{\Omega^-} \mathcal{A} ^{jk} \frac{\partial v^-}{\partial x_j} \cdot  \frac{\partial \phi}{\partial x_k} \, dx
=& \int_{\Omega^+} \operatorname{curl} _\eta \mathcal{C}(v)  \cdot \phi \det \nabla \eta\,  dx\nonumber\\
& \qquad +  \int_{\partial\Omega^+} \mathcal{C}(v) \times (a^T N) \, \phi\,  dS(0)\,,
\end{align*}
which is the variational form of the general elliptic system (\ref{vector-valued_elliptic_eq}) studies in Section \ref{sec::elliptic_regularity},  with
forcing functions
\begin{align*}
{\bf f_-}=0\,,\ \ 
{\bf f_+} = \operatorname{curl} _\eta  \mathcal{C} (v) \, \det \nabla \eta \,,\ \ \text{ and } \ \ 
{\bf g}= \mathcal{C} (v) \times (a^T N) \,,
\end{align*}
for which our regularity result Theorem \ref{thm:vector-valued_elliptic_eq_Sobolev_coeff} applies.
We therefore have that (for $k\ge 2$)
\begin{align} 
& \|\bar v^+\|_{H^{k+1}(\Omega^+)}  +  \|\bar v^-\|_{H^{k+1}(\Omega^-)}   \nonumber \\
& \qquad
 \le C \Big[\|\f_\pm\|_{H^{k-1}(\Omega^\pm)}+ \|\g\|_{H^{k-0.5}(\Gamma)}  
+ \P\big(\|\mathcal{A}_\pm\|_{H^\rk(\Omega^\pm)}\big) \Big(\|\f_\pm\|_{L^2(\Omega^\pm)} + \|\g\|_{H^{-0.5}(\Gamma)}\Big) \Big]\,, \label{nncs1}
\end{align} 
where $\P$ is a polynomial function and the constant 
 $C$ depends on $\Omega^\pm$.

 From (\ref{fp6.0}) and for 
\begin{equation}
\label{2309.1} 
 \sqrt{T}\ M\le \epsilon_0\,,
 \end{equation}
 with $0 < \epsilon_0\ll 1$ denoting a sufficiently small constant  (which is independent of $M$), that for any $v\in{\bf V}_M^k$ 
 \begin{equation}
 \label{2309.2}
 \|\eta\|_{H^{k+1}(\Omega^+)}\le C|\Omega|\,.
 \end{equation}
 Since from the definition (\ref{fp8}),
 \begin{equation}
 \label{2309.3}
 \|\mathcal{C}(v)\|_{H^{k}(\Omega^+)}\le \|u_0\|_{H^{k+1}(\Omega^+)}(1+C\sqrt{T} M)\,,
 \end{equation}
 we then infer from (\ref{2309.3}), (\ref{2309.2}) and (\ref{nncs1}) that
 \begin{equation*} 
 \|\bar v^+\|_{H^{k+1}(\Omega^+)}  +  \|\bar v^-\|_{H^{k+1}(\Omega^-)}   
\le C \Big[ C|\Omega|\|u_0\|_{H^{k+1}(\Omega^+)}(1+C\epsilon_0) 
(1+ \P(|\Omega|))  \Big]\,.
\end{equation*} 
Therefore,
\begin{equation*} 
 \|\bar v^+\|_{L^2(0,T;H^{k+1}(\Omega^+))}  +  \|\bar v^-\|_{L^2(0,T;H^{k+1}(\Omega^-))}   
\le 2 C \Big[ C|\Omega|\|u_0\|_{H^{k+1}(\Omega^+)}(1+C\epsilon_0)
(1+ \P(|\Omega|))  \Big]\sqrt{T}\,,
\end{equation*} 
which thanks to (\ref{2309.1}) shows that
\begin{equation*} 
 \|\bar v^+\|_{L^2(0,T;H^{k+1}(\Omega^+))}  +  \|\bar v^-\|_{L^2(0,T;H^{k+1}(\Omega^-))}   
\le 2 C \Big[ C|\Omega|\|u_0\|_{H^{k+1}(\Omega^+)}(1+C\epsilon_0)
(1+ \P(|\Omega|))  \Big]\frac{\epsilon_0}{M}\,.
\end{equation*}
  This inequality  
  then  
  proves that $\bar{v}\in{\bf V}_M^k$ for $$M^2=2 C \Big[ C|\Omega|\|u_0\|_{H^{k+1}(\Omega^+)}(1+C\epsilon_0)
(1+ \P(|\Omega|))  \Big]\epsilon_0\,.$$ 
 
Moreover, it is easy to check that the map  $\Theta:\ v \mapsto \bar{v}$ is sequentially weakly lower semi-continuous;  that is,  if 
$v_j\rightharpoonup v$ in the weak topology of the norm defining the closed convex set ${\bf V}_M^k$,
 then $\Theta v_j \rightharpoonup  \Theta v$.  Therefore, by Schauder's second fixed-point theorem (see \cite{Zeidler1986}, page 452), which is itself a corollary of Tyhonov's fixed-point theorem, we then have that $\Theta$ has a fixed point in ${\bf V}_M^k$.
 
\subsubsection{The fixed point is a solution to the Euler equations}
 We now explain why this fixed point, $v=\bar{v}$,  is  indeed a solution of the Euler equations with initial data $u_0$, and hence a solution to the
 3-D vortex patch boundary.
 At a fixed point $v=\bar v$,   (\ref{fp9}) becomes the following variational problem:
\begin{equation}
 \int_{\Omega^+} \mathcal{A} ^{jk} \frac{\partial v^+}{\partial x_j} \cdot  \frac{\partial \phi}{\partial x_k}  \, dx 
 +  \int_{\Omega^-} \mathcal{A} ^{jk} \frac{\partial v^-}{\partial x_j} \cdot  \frac{\partial \phi}{\partial x_k} \, dx= \int_{\Omega^+} \mathcal{C}(v)\cdot 
 \operatorname{curl} _\eta \phi \  \det \nabla \eta \ dx \ \ \forall  \phi\in H^1(\Omega) \,, \label{ncs2}
\end{equation}
where the operator $ \operatorname{curl} _\eta$ is defined (\ref{curleta}).
We define the following Eulerian quantities associated to our Lagrangian velocity $v$ and test function $\phi$:
$$u = v \circ \eta ^{-1}  \,, \ \  \mathfrak{C} = \mathcal{C} (v) \circ \eta ^{-1} \,, \ \ \text{ and } \ \ w = \phi \circ \eta^{-1}  \,.$$
The
change-of-variables theorem shows that (\ref{ncs2}) can be written as\footnote{Note that $\eta(\Omega,t)$ is the image of the $2\ell$-periodic box, and hence functions defined on $\eta(\Omega,t)$ are periodic.}
\begin{equation}
\label{why1}
\int_{\eta(\Omega^+,t)} \nabla u^+ \cdot\nabla w \ dy + \int_{\eta(\Omega^-,t)} \nabla u^- \cdot\nabla w \ dy  =  \int_{\eta(\Omega^+,t)}  \mathfrak{C} \cdot\operatorname{curl} w \  dy\,.
\end{equation}

Our goal is to show that $ \operatorname{div} u =0$ and that $ \operatorname{curl} u = \mathfrak{C} $.   To do so, we use 
 integration-by-parts on the left-hand side of (\ref{why1}); we see that
\begin{align*} 
 \int_{\eta(\Omega,t)} \nabla u \cdot\nabla w \ dy & =  - \int_{\eta(\Omega,t)}  \Delta  u \cdot w \ dy  + \int_{\eta(\Gamma,t)}  \jump{ \nabla _n u} \cdot w \, dS(t) \\
 & =   \int_{\eta(\Omega,t)}   \operatorname{curl} \operatorname{curl}   u \cdot w \ dy 
 -  \int_{\eta(\Omega,t)}  \nabla \operatorname{div}   u \cdot w \ dy  + \int_{\eta(\Gamma,t)}  \jump{ \nabla _n u} \cdot w \, dS(t) \\
 & =   \int_{\eta(\Omega,t)}   \operatorname{curl}   u \cdot\operatorname{curl} w \ dy 
 +  \int_{\eta(\Omega,t)}  \operatorname{div}   u \cdot \operatorname{div} w \ dy   \\
 & \qquad \qquad 
 + \int_{\eta(\Gamma,t)}  \left( \jump{\nabla _n u}  + \jump{n \times \operatorname{curl} u} -  \jump{ \operatorname{div} u}\, n \cdot\right)   \cdot w \, dS(t) \,.
\end{align*} 
The identities (\ref{cs1.10.12}) and (\ref{cs2.10.11}) show that for $u \in H^1(\eta(\Omega,t))$, so that $\jump { u}=0$ on $\eta(\Gamma,t)$, we have
that
$$
 \jump{\nabla _n u}  + \jump{n \times \operatorname{curl} u} -  \jump{ \operatorname{div} u}\, n =0  \ \text{ on } \ \eta(\Gamma,t)\,,
$$
so that
\begin{equation}
\label{why2}
\int_{\eta(\Omega,t)} \nabla u\cdot\nabla w\ dy =  \int_{\eta(\Omega,t)}[ \operatorname{curl} u \cdot \operatorname{curl} w+ \operatorname{div} u\  \operatorname{div} w] \ dy \,.
\end{equation}
Comparing (\ref{why1}) and (\ref{why2}), we have that for all test function $w \in H^1(\eta(\Omega,t) )$, 
\begin{equation}\label{why100}
\int_{\eta(\Omega,t)}[ \operatorname{curl} u \cdot \operatorname{curl} w+ \operatorname{div} u\  \operatorname{div} w] \ dy =  \int_{\eta(\Omega,t)}  {\bf 1}_{\eta(\Omega^+,t)}\mathfrak{C} \cdot\operatorname{curl} w \  dy \,.
\end{equation} 

We now chose the test function $w$ to have the potential form
\begin{equation*}
w=\nabla\psi\,,
\end{equation*}
for some function periodic function $\psi \in H^2(\eta(\Omega,t))$. Then,
\begin{equation}
\nonumber
\operatorname{curl} w=0\,,
\end{equation}
 and (\ref{why100}) reduces to
\begin{equation}
\label{why4}
 \int_{\eta(\Omega,t)} \operatorname{div} u\  \Delta\psi\ dx =0 \,.
\end{equation}
Since $u\in H^1(\eta(\Omega,t))$ and is periodic, there exists a periodic function $\psi_0\in H^2(\eta(\Omega,t))$,  such that 
\begin{equation}
\operatorname{div} u=\Delta \psi_0 \text{ in } \eta(\Omega,t)\,.
\end{equation}
Letting $\psi = \psi_0$   in (\ref{why4}) then shows that
\begin{equation*}
0=  \int_{\eta(\Omega,t)} (\operatorname{div} u)^2\ dx\,,
\end{equation*}
and thus
\begin{equation}
\label{why5}
\operatorname{div} u=0\,.
\end{equation}
This being true for all time $t\in [0,T]$, since $\eta(x,0)=x$,  we then infer that
\begin{equation}
\label{why6}
\text{det}\nabla\eta=1\,.
\end{equation}
Using (\ref{why6}) and (\ref{why5}) in (\ref{why100}), we see that for all $w \in H^1(\eta(\Omega,t))$,
\begin{equation}
\label{why7}
0=  \int_{\eta(\Omega,t)} \left(  \mathfrak{C} -\operatorname{curl} u\right)\cdot \operatorname{curl} w \ dy \,,
\end{equation}
and 
from (\ref{fp8}) we see that $ \mathcal{C} (v) (x,t) =0$ for all $x \in \Omega^-$, since $\omega_0^-=0$.   Next, we note that
\begin{align*} 
\partial_t \mathcal{C} (v)   &= \frac{\p v}{\p x_k} \omega_0^k 
 = \frac{\p v}{ \p x_r} A^r_j \frac{ \p \eta^j}{ \p x_k} \omega_0^k  = \frac{\p v}{ \p x_r} A^r_j  \mathcal{C} (v)^j  = \nabla u(\eta) \cdot \mathcal{C} (v) \,,
\end{align*} 
where $ \nabla u( \eta ) $ denotes $ \nabla u \circ \eta$.   Hence, since $ \mathcal{C} (v) = \mathfrak{C}\circ \eta$, it follows that $ \mathfrak{C} $
satisfies 
\begin{equation}
\label{why10.2}
\mathfrak{C}_t+ \nabla _u \mathfrak{C} -\nabla u\cdot  \mathfrak{C} = 0 \ \text{ in } \ \eta(\Omega^+,t) \,,
\end{equation}
and $ \mathfrak{C} (y,t) =0$ for all $y \in  \eta(\Omega^-,t)$.
Since $\mathfrak{C} \in H^k(\eta(\Omega^+,t))$, $k \ge 2$, we take the divergence of equation (\ref{why10.2}) and find that
\begin{equation}
\operatorname{div} \mathfrak{C}_t+ \nabla _u\operatorname{div}\mathfrak{C}- \nabla _{\mathfrak{C}}\operatorname{div} u + (u^i,_j \mathfrak{C}^j,_i-u^j,_i\ \mathfrak{C}^i,_j)=0\,
\end{equation}
From (\ref{why5}) and  the symmetry of the last two terms, we conclude that
\begin{equation*}
\operatorname{div} \mathfrak{C}_t+ \nabla _u \operatorname{div}\mathfrak{C}=0\,,
\end{equation*}
and thus
\begin{equation}
\label{why11}
\operatorname{div} \mathfrak{C}(\eta(x,t),t)=\operatorname{div} \mathfrak{C}(x,0)\,.
\end{equation}
Since $\mathfrak{C}(0)=\operatorname{curl}u_0$ we then have from (\ref{why11}) that
\begin{equation}
\label{why12}
\operatorname{div} \mathfrak{C}(\eta(x,t),t)=0\,.
\end{equation}

From (\ref{transport2}) and (\ref{why6})
$$
\int_{\eta(\Omega^+,t)} \mathfrak{C} (y,t) \ dy =0\,
$$
We note that $  \mathfrak{C} ( \cdot ,t) \in L^2(\eta(\Omega,t))$. Next, we
 define the periodic vector-field $\psi \in H^2(\eta(\Omega,t))$ as the solution, modulo constants, of
\begin{alignat*}{2}
-\Delta\psi^+=& \mathfrak{C}  && \ \text{ in }\ \eta(\Omega^+,t)\,,\\
-\Delta\psi^-=&0 &&\  \text{ in }\ \eta(\Omega^-,t)\,,
\end{alignat*}
with the continuity conditions, which follow from the fact that $ \nabla \psi \in H^1(\eta(\Omega,t))$, 
\begin{equation}\label{new_bcs}
\jump {\psi} =0 \ \text{ and } \jump { \nabla _n \Psi} =0 \ \text{ on } \eta(\Gamma,t) \,.
\end{equation} 
Theorem \ref{thm:vector-valued_elliptic_eq_Sobolev_coeff} shows that
 $\psi\in H^{k+2}(\eta(\Omega^\pm,t))$, $k \ge 2$.
Moreover,  from (\ref{why11}),  $\operatorname{div}\psi$ is harmonic in  both $\eta(\Omega^+,t)$ and $\eta(\Omega^-,t)$ and is a 
periodic function;  furthermore, $\jump{\nabla\operatorname{div}\psi\cdot n}=0$ on $\Gamma(t)$, for
\begin{align}
\nabla\operatorname{div}\psi^\pm\cdot n=&\operatorname{curl}(\operatorname{curl}\psi^\pm)\cdot n+\Delta\psi^\pm\cdot n\nonumber\\
=&\operatorname{curl}(\operatorname{curl}\psi^\pm)\cdot n+ \mathfrak{C} \cdot n\nonumber\\
=&\operatorname{curl}(\operatorname{curl}\psi^\pm)\cdot n \nonumber \\
=&\nabla_{\tau_1}(\operatorname{curl}\psi^\pm)\cdot\tau_2-\nabla_{\tau_2}(\operatorname{curl}\psi^\pm)\cdot\tau_1 \nonumber \,, 
\end{align}
where we have used $ \mathfrak{C} \cdot n=0$ on $\Gamma(t)$ in the third equality, 
so that 
\begin{equation}
\jump {\nabla\operatorname{div}\psi\cdot n}
= \jump{\nabla_{\tau_1}\operatorname{curl}\psi\cdot\tau_2
-\nabla_{\tau_2}\operatorname{curl}\psi\cdot\tau_1 } \label{normaltrace3}\,.
\end{equation}
Using (\ref{new_bcs}), 
$$
\jump {\operatorname{curl}\psi}=0 \ \text{ on } \eta(\Gamma,t) \,,
$$
so that
\begin{equation}
\nonumber
\jump{\nabla_{\tau_\alpha}\operatorname{curl}\psi}=0  \ \text{ on } \eta(\Gamma,t) \,,
\end{equation}
and from (\ref{normaltrace3}), 
\begin{equation}
\jump{\nabla\operatorname{div}\psi^\cdot n}
=0\ \text{ on } \eta(\Gamma,t) 
\label{normaltrace5}\,.
\end{equation}
We now set $\Omega^\pm(t) = \eta(\Omega^\pm,t)$ 
Using (\ref{normaltrace5}) and the fact that $\operatorname{div}\psi\in H^1(\Omega(t))\cap H^{k+1}(\Omega^\pm(t))$, $k \ge 2$, is harmonic in 
$\Omega^\pm(t)$  and is a periodic function, we find that
\begin{align*}
0=&\int_{\Omega^+(t)}\Delta\operatorname{div}\psi\ \operatorname{div}\psi\ dy+ \int_{\Omega^-(t)}\Delta\operatorname{div}\psi\ \operatorname{div}\psi\ dy\,\\
=&-\int_{\Omega^+(t)}|\nabla\operatorname{div} \psi|^2 dy
-\int_{\Omega^-(t)}|\nabla \operatorname{div} \psi|^2 dy+  \int_{\Gamma(t)} \jump {\nabla\operatorname{div}\psi^\cdot n}\operatorname{div}\psi\ dS(t) \,\\
=&-\int_{\Omega^+(t)}|\nabla \operatorname{div} \psi|^2 dy-\int_{\Omega^-(t)}|\nabla\operatorname{div} \psi|^2 dy \,
\end{align*}
which shows that $\operatorname{div}\psi ( \cdot , t) $ is a constant.

 Therefore,
\begin{equation}
\label{why13}
\Delta\psi=-\operatorname{curl}(\operatorname{curl}\psi)\,,
\end{equation}
so that $ -\operatorname{curl}(\operatorname{curl}\psi) = \mathfrak{C} $.  Substituting this into 
(\ref{why7}), we see that  for all test functions $w \in H^1(\eta(\Omega,t)) $,
\begin{equation}
\label{why14}
0=  \int_{\eta(\Omega,t)} \left( -\operatorname{curl}(\operatorname{curl}\psi) -\operatorname{curl} u\right)\cdot\operatorname{curl} w\  dy\,.
\end{equation}
Next, we set $w=-\operatorname{curl}\psi+u$ in (\ref{why14}), which satisfies the condition of being a test function, and obtain that
\begin{equation*}
0=  \int_{\eta(\Omega,t)} |\operatorname{curl} (\operatorname{curl}\psi+u)|^2\ dx\,,
\end{equation*}
and thus
\begin{equation*}
\operatorname{curl} u^+=-\operatorname{curl}(\operatorname{curl}\psi^+)= \mathfrak{C}   \text{ in } \Omega^+(t) \,,
\end{equation*}
and
\begin{equation*}
\operatorname{curl} u^-=-\operatorname{curl}(\operatorname{curl}\psi^-)= 0  \, \text{ in } \Omega^-(t) \,.
\end{equation*}
Thanks to (\ref{why10.2}), we have that
in $\Omega^+(t)$,
\begin{equation*}
\operatorname{curl} u_t+ \nabla _u \operatorname{curl}u -\nabla u \cdot \operatorname{curl} {u}=0\,,
\end{equation*}
which is the same as
\begin{equation*}
\operatorname{curl} \left( u_t + \nabla _uu\right )=0\,,
\end{equation*}
from which we infer the existence of a pressure function $p$ such that
\begin{equation*}
 u_t+ \nabla _u u +\nabla p =0\,.
\end{equation*}
Therefore,  $u$ is solution of the incompressible Euler equations  (\ref{euler}), as we have already proven that $ \operatorname{div} u=0$.

It remains only to show that $u(x,0) = u_0(x)$.  To this end,  we notice that from (\ref{fp8}),
\begin{equation*}
\mathcal{C}(v)(x, 0)=\operatorname{curl} u_0(x)\,,
\end{equation*}
and thus
\begin{equation*}
\operatorname{curl} u(\cdot , 0)=\operatorname{curl} u_0\,,
\end{equation*}
which coupled with the fact that  $\operatorname{div}u( \cdot , 0)=0=\operatorname{div}u_0$ and the periodicity of $u$, provides us with
\begin{equation*}
u(\cdot , 0)=u_0+c\,,
\end{equation*}
where $c$ is a constant vector. From (\ref{fp9.2}),  
\begin{equation*}
\int_{\Omega} u(x , 0)\ dx=0\,,
\end{equation*}
which coupled with 
\begin{equation*}
\int_{\Omega} u_0(x)\ dx=0\,,
\end{equation*}
then shows that $c=0$, so that $u(\cdot , 0)=u_0$, which completes our proof that $u$ is solution of the vortex patch problem on $[0,T]$, with
the desired regularity properties.   In particular, by (\ref{fp2}) and (\ref{2309.2}), we see that $\eta \in \mC^0([0,T]; H^{k+1}(\Omega^+))$ and hence
by the trace theorem, $\eta \in \mC^0([0,T]; H^{k+1/2}(\Gamma))$.   Since the vortex patch boundary $\Gamma(t) = \eta(\Gamma,t)$ for each
$t\in [0,T]$, we see that $\Gamma(t)$ is of Sobolev class $H^{k+1/2}$.  To explain why $\Gamma(t)$ is indeed $\eta(\Gamma,t)$,  we use
the identity $ \operatorname{curl} u \circ \eta = \nabla \eta \cdot \omega_0$, where we recall that $\omega_0 = \operatorname{curl} u_0$ and
satisfies (\ref{vp3d}).  Next, we choose a local coordinate system at a point $x \in \Gamma$, such that $n(x, 0) = e_3$ and the two tangent
vectors are $\tau_1=e_1$ and $\tau_2=e_2$.    By conditions (\ref{vp3d}g,h), we can write 
$\omega_0^+ = \sum_{ \alpha =1}^2 \omega_0^+ \cdot  e_\alpha \, e_\alpha $.  This means that 
$ \operatorname{curl} u^+(\eta(x,t),t) = \eta,_ \alpha \omega_0^+ \cdot e_ \alpha $, and as we have shown already, 
 $ \operatorname{curl} u^+(\eta(x,t),t) 
\cdot n(\eta(x,t),t) =  \omega_0^+ \cdot e_ \alpha  \eta,_ \alpha \cdot \frac{(\eta,_1 \times \eta,_2)}{|\eta,_1 \times \eta,_2|}=0$.
 Since for $\alpha =1,2$, $\eta,_ \alpha $ is a tangent vector
to $\eta(\Gamma,t)$ at the point $ \eta(x,t)$ and hence continuous, then
$$
\jump{ \operatorname{curl} u } \circ \eta = \eta,_ \alpha \jump {  \omega_0^+ \cdot e_ \alpha  }\,.
$$
This shows that the set $\Gamma(t)$, on which $ \operatorname{curl} u( \cdot  ,t)$ has a jump discontinuity, is propagated by the Lagrangian
flow map $\eta( \cdot, t)$.

   Uniqueness of solutions has been shown by Gamblin \& Saint Raymond \cite{GaSa1995}.

\section{Elliptic Regularity}\label{sec::elliptic_regularity}
\subsection{A two-phase elliptic problem}\label{sec:vector-valued_elliptic_eq}

For $k \ge 2$, 
let $\Omega^+\subseteq \Rn$ denote an open, bounded $H^{k+1} $-domain which is diffeomorphic to the unit ball $B=
\{x \in \mathbb{R}^\n  \ : \  |x| < 1\}$.  We set $\Gamma := \partial \Omega^+$, which is then an $H^{k+1/2}$-class closed surface. Let $\Omega$ denote a periodic box $[-\mathcal{L} ,\mathcal{L} ]^\n$ in $ \mathbb{R}^\n  $ with opposite sides identified, and with $\mathcal{L} $ sufficiently large so that $ \overline{\Omega^+}$ is properly contained in $\Omega$.
 Functions defined on $\Omega$ are $2 \mathcal{L} $-periodic in each of the $\n$ coordinate directions, i.e.,
$$
u(x + 2\ell e_i) = u(x) \ \ \forall x \in \mathbb{R}^\n  , i=1,...,\n \,,
$$
were $e_i$ denotes the usual Cartesian basis.
We set
 $\Omega^- =  \Omega/\overline{\Omega^+}^c$.

We establish elliptic regularity for the following two-phase vector-valued elliptic problem:
\begin{subequations}
\label{vector-valued_elliptic_eq}
\begin{alignat}{2}
 - \frac{\p}{\p x_j} \Big(a_\pm^{jk} \frac{\p \u_\pm}{\p x_k}\Big) &= \f_\pm \qquad&&\text{in}\quad\Omega^\pm\,, \\
\jump{\u} &= 0 \qquad&&\text{on}\quad\Gamma \,,\\
\Bigjump{a^{jk} \frac{\p \u}{\p x_k} N_j} &=\g  \qquad&&\text{on}\quad \Gamma\, \\
\u_- &\text{ is  periodic}  \qquad&&\text{on}\quad\partial \Omega\,
\end{alignat}
\end{subequations}
where $\u_\pm = (\u_\pm^1,\cdots,\u_\pm^\n)$ and $\f_\pm = (\f^1,\cdots,\f^\n)$, $\g = (\g^1,\cdots,\g^\n)$ are vector-valued functions, and $a_\pm^{jk}$ are
 two-tensors which  satisfy the positivity condition
\begin{equation}\label{positivity_of_a}
a_\pm^{jk} \xi_j \xi_k \ge \lambda |\xi|^2 \qquad\Forall \xi \in \Rn
\end{equation}
for some $\lambda > 0$. We use the notation $\jump{\w} = \w_+ - \w_-$ for vector fields $\w$ on $\Gamma$, and we let
$N$ denote the outward unit normal to $\partial \Omega^+$.     The system (\ref{vector-valued_elliptic_eq}) has a unique solution in $H^1(\Omega)$
when we additionally assume that $\int_\Omega u(x) dx =0$.

\def\intg{\int_{\Gamma}}
\def\into{\int_{\Omega}}
\def\intp{\int_{\Omega^+}}
\def\intm{\int_{\Omega^-}}

Let $\V = H^1( \Omega)$, the space of $H^1$ functions on $[- \mathcal{L} ,\mathcal{L} ]^\n$ which are $2\mathcal{L} $-periodic.  Let
 $\u = \u_+ {\bf 1}_{\overline{\Omega^+}} + \u_- {\bf 1}_{\Omega^-}$,
$\f = \f_+ {\bf 1}_{\overline{\Omega^+}} + \f_- {\bf 1}_{\Omega^-}$. The variational (or weak) form of
(\ref{vector-valued_elliptic_eq}) is given by
\begin{equation}\label{vector-valued_elliptic_weak_form}
 \int_{\Omega^\pm} \hspace{-1pt}a^{jk} \frac{\p \u^i}{\p x_k} \frac{\p \Varphi^i}{\p x_j}\, dx = \int_{\Omega^\pm} \f\,\Varphi\, dx 
 + \intg \g \,  \Varphi \, dS \quad\Forall \Varphi \in \V\,,
\end{equation}
where we use the following integral notation:
$$
 \int_{\Omega^\pm} \hspace{-1pt}a^{jk} \frac{\p \u^i}{\p x_k} \frac{\p \Varphi^i}{\p x_j}\, dx =
  \int_{\Omega^+} \hspace{-1pt}a_+^{jk} \frac{\p \u_+^i}{\p x_k} \frac{\p \Varphi^i}{\p x_j}\, dx
+  \int_{\Omega^-} \hspace{-1pt}a_-^{jk} \frac{\p \u_-^i}{\p x_k} \frac{\p \Varphi^i}{\p x_j}\, dx \,
$$
and 
$$
\int_{\Omega^\pm} \f\,\Varphi\, dx = \int_{\Omega^+} \f_+\,\Varphi\, dx + \int_{\Omega^-} \f_-\,\Varphi\, dx \,.
$$

 The regularity theory for solutions $\u$ of (\ref{vector-valued_elliptic_weak_form}) is classical when the coefficient matrix $ a^{jk}$ is
 in $\mC^k$, and can be summarized by the following
 \begin{theorem}\label{thm:vector-valued_elliptic_regularity}
Suppose that for some $\rk\in \bbN$,  $a_\pm^{jk} \in \mC^\rk(\cls{\Omega^\pm})$ satisfies {\rm(\ref{positivity_of_a})}. Then for all $\f_\pm\in H^{\rk-1}(\Omega^\pm)$ and $\g \in H^{\rk-0.5}(\Gamma)$, the solution $\u$ to {\rm(\ref{vector-valued_elliptic_eq})} is in $H^{\rk+1}(\Omega^\pm)$, and satisfies
\begin{equation}\label{vector-valued_elliptic_regularity}
\|\u_\pm\|_{H^{\rk+1}(\Omega^\pm)}   \le C \Big[\|\f_\pm\|_{H^{\rk-1}(\Omega^\pm)}
 + \|\g\|_{H^{\rk-0.5}(\Gamma)} \Big]
\end{equation}
for some constant $C$ depending on $\|a_\pm\|_{\mC^\rk(\Omega^\pm)}$.
\end{theorem}
We use the following notation for norms:
$$\|( \cdot ) _\pm\|_{H^{\rk+1}(\Omega^\pm)} = \|(\cdot )_+\|_{H^{\rk+1}(\Omega^+)} 
+ \|( \cdot )_-\|_{H^{\rk+1}(\Omega^-)} \,.
$$

We shall need the corresponding result for the case that the coefficient matrix $a_\pm^{jk}$ has only Sobolev-class regularity:

\begin{theorem}\label{thm:vector-valued_elliptic_eq_Sobolev_coeff}
 Suppose that for some integer $\rk > \novertwo$ and $1\le \ell \le \rk$\,, $a_\pm^{jk} \in H^\rk(\Omega^\pm)$ satisfies {\rm(\ref{positivity_of_a})}.
 Then if $f\in H^{\ell-1}(\Omega^+)$ and $g \in H^{\ell-0.5}(\Gamma)$, the weak solution $\u_\pm$ to {\rm(\ref{vector-valued_elliptic_eq})} is in
  $H^{\ell+1}(\Omega^\pm)$, and satisfies
\begin{align} 
& \|\u_\pm\|_{H^{\ell+1}(\Omega^\pm)}   \nonumber \\
& \qquad
 \le C \Big[\|\f_\pm\|_{H^{\ell-1}(\Omega^\pm)}+ \|\g\|_{H^{\ell-0.5}(\Gamma)}  
+ \P\big(\|a_\pm\|_{H^\rk(\Omega^\pm)}\big) \Big(\|\f_\pm\|_{L^2(\Omega^\pm)} + \|\g\|_{H^{-0.5}(\Gamma)}\Big) \Big]\,,
\label{vector-valued_elliptic_regularity_sobolev1}
\end{align} 
where $\P$ is a polynomial function and the constant 
 $C$ depends on $\Omega^\pm$. 
\end{theorem}
We are using the notation
\begin{align*} 
&
 \P\big(\|a_\pm\|_{H^\rk(\Omega^\pm)}\big) \Big(\|\f_\pm\|_{L^2(\Omega^\pm)} + \|\g\|_{H^{-0.5}(\Gamma)}\Big) \\
 &\
 =
  \P\big(\|a_+\|_{H^\rk(\Omega^+)}\big) \Big( \|\f_+\|_{L^2(\Omega^+)} + \|\g\|_{H^{-0.5}(\Gamma)}\Big)
  +
  \P\big(\|a_-\|_{H^\rk(\Omega^-)}\big) \Big( \|\f_-\|_{L^2(\Omega^-)} + \|\g\|_{H^{-0.5}(\Gamma)}\Big).
\end{align*} 
\begin{proof} 
Let $\rE^\pm: H^{\rk+1}(\Omega^\pm) \to H^{\rk+1}(\Rn)$ denote a Sobolev extension operator, and let ${a_\pm}_\epsilon 
= \eta_\epsilon \cvl 
(\rE^\pm a_\pm)$ and $\f_{\hspace{-1.5pt}\epsilon} = \eta_\epsilon \cvl (\rE^\pm f)$.
 Let $\{ \U_m\}_{m=1}^K$ denote an open cover of $\Omega$ which intersects the interface $\Gamma$, and let $\{ \theta_m\}_{m=1}^K$ denote a collection of charts such that
\begin{enumerate}
\item $\theta_m: B(0, r_m) \to \U_m \text{ is an } H^{k+1}\text{-diffeomorphism} $,
\item  $\det (\nabla \theta_m) >0 $,
\item $\theta_m: B^0_m \equiv B(0,r_m) \cap \{ x_n=0\} \to \U_m \cap \Gamma $,
\item $\theta_m: B^+_m \equiv B(0,r_m) \cap \{y_\n>0\} \to \U_m \cap \Omega^+ $,
\item $\theta_m: B^-_m \equiv B(0,r_m) \cap \{y_\n<0\} \to \U_m \cap \Omega^- $,
\item $\|\nabla \theta_m - \id\|_{L^\infty(B(0,r_m))} \ll 1$.
\end{enumerate}
Let $0 \le \zeta_m \le 1$ in $\mC^\infty_\cptspt(\U_m)$ denote a partition of unity subordinate to the open covering $\U_m$; that is,
$$
\sum\limits_{m=0}^K \zeta_m = 1 \quad\text{and}\quad \supp(\zeta_m) \subseteq \U_m \quad\Forall m\,.
$$
Finally, let $\g_\epsilon$ denote a smooth regularization of $\g$ defined by
\begin{align*}
\g_\epsilon &= \sum_{m=1}^K \sqrt{\zeta_m}\, \big[\Lambda_\epsilon \big((\sqrt{\zeta_m}\, \g)\circ\theta_m\big)\big]\circ \theta_m^{-1} \,.
\end{align*}
It follows that for  $\epsilon \ll 1$ sufficiently small, 
\begin{equation}\label{uniform_ellipticity_of_tilde_a}
{a_\pm^{jk}}_\epsilon(x) \xi_j \xi_k \ge \frac{\lambda}{2}\, |\xi|^2 \qquad\Forall\xi \in \Rn, x\in \Omega\,.
\end{equation}
Hence, by Theorem \ref{thm:vector-valued_elliptic_regularity}, the solution $\u^\epsilon $ to the variational problem
\begin{equation}\nonumber
 \int_{\Omega^\pm} \hspace{-1pt}a_ \epsilon ^{jk} \frac{\p {\u^ \epsilon } ^i}{\p x_k} \frac{\p \Varphi^i}{\p x_j}\, dx = \intp \f_ \epsilon\, \Varphi\, dx 
 + \intg \g_\epsilon  \,  \Varphi \, dS \quad\Forall \Varphi \in \V\,,
\end{equation}
satisfies $\u^\epsilon_\pm \in H^k(\Omega^\pm)$ for all $k \ge  1$; in particular, the vector fields $\u^\epsilon_\pm$ are smooth.
 We next establish an $\epsilon$-independent upper bound for 
$\|\u^\epsilon_\pm\|_{H^{\ell+1}(\Omega^\pm)}$.

\vspace{.1 in}

\noindent 
{\bf Step 1: Regularity in horizontal directions near $\Gamma$.}    
We fix $m \in \{1, ..., K\}$ and set  
$$U_\pm = \u_\pm^\epsilon \circ \theta_m \,, \ F= \f_ \epsilon  \circ \theta _m\,, \ G=\g_ \epsilon  \circ \theta _m \,, \ \xi = \zeta_m \circ \theta_m\,, 
\text{ and }  \Phi = \Varphi \circ \theta _m \,.
$$
With $A = [\nabla \theta_m]^{-1}$, we define $b^{rs} = (a^{jk}\circ\theta_m) A^s_k A^r_j$. Then, since
 $\|\nabla \theta_m - \id\|_{L^\infty(B^+_m)} \ll 1$, the matrix $b$ is positive-definite:
\begin{equation}\label{brs_elliptic}
b^{rs} \xi_r \xi_s = (a^{jk}\circ\theta_m) A^s_k A^r_j \xi_r \xi_s \ge \lambda |A^\rT \xi|^2 \ge \frac{\lambda}{4}
 |\xi|^2\ \ \forall \xi \in \mathbb{R}  ^\n\,.
\end{equation}
By the change-of-variables formula, the variational formulation is written as
\begin{equation}\nonumber
 \int_{B_m^\pm}  b^{rs} \frac{\p U^i}{\p x_r} \frac{\p \Phi^i}{\p x_s}\, dx 
 = \int_{B_m^\pm} F\, \Phi\, dx 
 + \int_{B_m^0} G  \,  \Phi \, dS \quad\Forall \Phi \in H^1_0(B_m)\,,
\end{equation}
where $\int_{B_m^\pm}  b^{rs} \frac{\p U^i}{\p x_r} \frac{\p \Phi^i}{\p x_s}\, dx =\int_{B_m^+}  b_+^{rs} \frac{\p U_+^i}{\p x_r} \frac{\p \Phi^i}{\p x_s}\, dx +\int_{B_m^-}  b_-^{rs} \frac{\p U_-^i}{\p x_r} \frac{\p \Phi^i}{\p x_s}\, dx$.

With  $ \Delta_0 = \sum_{ \alpha =1}^{\n-1}  \frac{\p^2}{ \p x_{ \alpha } ^2}$ denoting the horizontal Laplace operator,   we define the
test function
$$
\Phi= (-1)^\ell \big[ \xi  \Delta _0^{\ell} ( \xi U)\big] \,,$$
 so that
\begin{equation}\label{pf_of_thm12.12_id1}
 \int_{B_m^\pm}  b^{rs} \frac{\p U^i}{\p x_r} \frac{\p \Phi^i}{\p x_s}\, dx 
 \le C \Big[\|\f_+\|_{H^{\ell-1}(\Omega^+)}+ \|\f_-\|_{H^{\ell-1}(\Omega^-)} + \|\g\|_{H^{\ell-0.5}(\Gamma)} \Big] \big\|\bp^\ell ( \xi U_\pm)\big\|_{H^1(B^\pm_m)} \,.
\end{equation}

We focus now on the left-hand side of (\ref{pf_of_thm12.12_id1}).  We let $\bp =( \p_1, \cdot \cdot\cdot , \p_{n-1})$ denote the horizontal gradient, 
and write
\begin{align*} 
 \bp ^\ell  V \, \bp ^\ell  W & =\sum_{\alpha _1=1}^{\n-1} \cdot\cdot\cdot \sum_{\alpha _{\ell}=1}^{\n-1}  \frac{\p^\ell V}{\p x_{ \alpha _1}\cdot\cdot\cdot 
 \p x_{ \alpha _\ell} } \  \frac{\p^\ell W}{\p x_{ \alpha _1}\cdot\cdot\cdot 
 \p x_{ \alpha _\ell} } \,, \\
  \bp ^{\ell-1}  V \, \bp ^{\ell+1}  W & =\sum_{\alpha _1=1}^{\n-1} \cdot\cdot\cdot \sum_{\alpha _{\ell-1}=1}^{\n-1}  \frac{\p^\ell V}{\p x_{ \alpha _1}\cdot\cdot\cdot 
 \p x_{ \alpha _{\ell-1}} } \  \frac{\p^\ell  \Delta _0 W}{\p x_{ \alpha _1}\cdot\cdot\cdot 
 \p x_{ \alpha _{\ell-1}} } \,, \\
\end{align*} 
and so forth.
Then,
\begin{align}
 \int_{B_m^\pm}  b^{rs} \frac{\p U^i}{\p x_r} \frac{\p \Phi^i}{\p x_s}\, dx 
 &= \int_{B^\pm_m} \bp^\ell \big[b^{rs} (\xi U),_r\big] \bp^\ell (\xi U ),_s dx
- \int_{B^\pm_m} \bp^\ell \big[b^{rs}U \xi ,_r\big] \bp^\ell (\xi U),_s\, dx \nonumber\\
&+ \int_{B^\pm_m} \bp^{\ell-1} \big[b^{rs}U,_r \xi ,_s\big] \bp^{\ell+1} (\xi U) dx \,.  \label{vector-valued_elliptic_proof_id1}
\end{align}
For the first term on the right-hand side of (\ref{vector-valued_elliptic_proof_id1}), we make use of  (\ref{brs_elliptic}) and
Young's inequality to conclude that
\begin{align*}
\int_{B^\pm_m} &\bp^\ell \big[b^{rs} (\xi U ),_r\big] \bp^\ell (\xi U ),_s dx 
 = \int_{B^\pm_m} b^{rs} \bp^\ell (\xi U),_r \bp^\ell (\xi U),_s dx 
+ \int_{B^\pm_m} \big[\comm{\bp^\ell}{b^{rs}}(\xi U),_r\big] \bp^\ell (\xi U ),_s dx \\
&\qquad \ge \big(\frac{\lambda}{8} - \delta\big) \big\|\bp^\ell \nabla (\xi U _\pm)\big\|^2_{L^2(B_m^\pm)} 
- C_\delta \big\|\comm{\bp^\ell}{b}\nabla (\xi U_\pm)\big\|^2_{L^2(B^\pm_m)}\,.
\end{align*}
Then, Corollary \ref{lem:useful_lemma_with_Sobolev_class_coeff1} with $\epsilon = 1/8$ shows that
\begin{align}
& \int_{B^\pm_m} \bp^\ell \big[b^{rs} (\xi U ),_r\big] \bp^\ell (\xi U ),_s dy 
\ge \big(\frac{\lambda}{8} - \delta\big) \big\|\bp^\ell \nabla (\xi U_\pm)\big\|^2_{L^2(B^\pm_m)} 
- C_\delta \|a_\pm\|^2_{H^\rk(\Omega_\pm)}\|\u_\pm^\epsilon\|^2_{H^{\ell+\frac{7}{8}}(\Omega_\pm)} \,. \label{ss1}
\end{align}
By Lemma \ref{prop:HkHl_product},
 for $0 \le \ell\le \rk+1$, $f_\pm\in H^{\max\{\rk,\ell\}}(\Omega_\pm)$ and $g_\pm\in H^\ell(\Omega_\pm)$,
and for a generic $C$,
\begin{equation}\label{HkHlproduct}
\begin{array}{l}
\displaystyle{}\|f_\pm \  g_\pm\|_{H^\ell(\Omega)} \le C \|f_\pm\|_{H^{\max\{\rk,\ell\}}(\Omega^\pm)} \|g_\pm\|_{H^\ell(\Omega^\pm)} \vspace{.2cm}\\
\qquad\qquad\qquad\quad \displaystyle{}\Forall f_\pm\in H^{\max\{\rk,\ell\}}(\Omega^\pm), g_\pm \in H^\ell(\Omega^\pm)\,.
\end{array}
\end{equation}
For the second and third terms on the right-hand side of (\ref{vector-valued_elliptic_proof_id1}), we 
use the inequality (\ref{HkHlproduct}), and find that
\begin{align}
&\left|  \int_{B^\pm_m} \bp^\ell \big[b^{rs}U \xi ,_r\big] \bp^\ell (\xi U),_s\, dx \right|
+ \left|\int_{B^\pm_m} \bp^{\ell-1} \big[b^{rs}U,_r \xi ,_s\big] \bp^{\ell+1} (\xi U) dx \right| \nonumber \\
&\qquad\qquad \qquad\qquad \quad \le C_\delta \|a_\pm\|^2_{H^\rk(\Omega)} \|\u_\pm^\epsilon\|^2_{H^\ell(\Omega^\pm)} 
+ \delta \big\|\bp^\ell \nabla (\xi U_\pm)\big\|^2_{L^2(B^\pm_m)}\,. \label{ss2}
\end{align}
 Choosing $\delta>0$ sufficiently small in (\ref{ss1}) and (\ref{ss2}), we conclude that
\begin{equation}\label{vector-valued_elliptic_est_temp2}
 \big\|  \xi \bp^\ell \nabla U _\pm\big\|_{L^2(B^\pm_m)}
\le C \Big[\|\f_\pm\|_{H^{\ell-1}(\Omega^\pm)} + \|\g\|_{H^{\ell-0.5}(\Gamma)} + \|a\|_{H^\rk(\Omega)} \|\u_\pm^\epsilon\|_{H^{\ell+\frac{7}{8}}(\Omega^\pm)} \Big]\,.
\end{equation}

\vspace{.1 in}
\noindent 
{\bf Step 2: Regularity in the vertical direction near $\Gamma$.} We write
(\ref{vector-valued_elliptic_eq}a) as
\begin{equation}\label{ss3}
-\xi  \big( b^{rs} {U_\pm},_s \big),_r = \xi  F_\pm \text{ in } B_m^\pm \,.
\end{equation} 
We analyze (\ref{ss3}) in the $+$-phase and drop the $+$-subscript for notational clarity.
With $U,_\n$ denoting $\p U/ \p x_\n$, we have that
\begin{equation}\label{pf_of_thm12.12_id2}
\begin{array}{l}
-\xi  b^{\n\n} U,_{\n\n} = \xi \Big[  F - b^{\n\n},_{\n} U,_{\n} - \sum_{(r,s)\ne (\n,\n)} b^{rs}, _{r} U,_s 
 - \sum_{(r,s)\ne (\n,\n)} b^{rs} U,_{sr} \Big] \text{ in } B^+_m\,.
\end{array}
\end{equation}
We analyze the terms on the right-hand side of (\ref{pf_of_thm12.12_id2}).
For any integer $j$ such that $0 \le j \le \ell -1$, 
\begin{align*}
 \big\|\bp^{\ell-1-j} \nabla^j F\big\|_{L^2(B^+_m)} 
 \le C \Big[ \|\f\|_{H^{\ell-1}(\Omega^+)} \Big]\,.
\end{align*}
Moreover, since $ \ell \le \rk$, by Lemma \ref{prop:HkHl_product} with $\epsilon = 1/8$,
\begin{align*}
& \big\|\bp^{\ell-1-j} \nabla^j ( \xi  b^{\n\n},_{\n} U,_{\n})\big\|_{L^2(B^+_m)} 
+ \sum_{(r,s)\ne (\n,\n)} \big\|\bp^{\ell-1-j} \nabla^j b^{rs},_{r} U,_s\big\|_{L^2(B^+_m)} \\
&\le C \sum_{r=0}^{\ell-1} \| \nabla ^{\ell-r} a D^{r+1} \u^\epsilon \|_{L^2(\Omega^+)} 
\le C \sum_{r=1}^{\ell} \|\nabla ^{\ell+1-r} a D^r \u^\epsilon \|_{L^2(\Omega^+)}
\le C_\epsilon \|a\|_{H^\rk(\Omega^+)} \|\u^ \epsilon \|_{H^{\ell+\frac{7}{8}}(\Omega^+)}\,.
\end{align*}
Finally, by Corollary \ref{lem:useful_lemma_with_Sobolev_class_coeff1} with $\epsilon = 1/8$,
\begin{align*}
\big\|\comm{\bp^{\ell-1-j}\nabla^j}{ \xi  b^{\n\n} } U,_{\n\n}\big\|_{L^2(B^+_m)} 
&+ \sum_{(r,s) \ne (\n,\n)} \big\| \comm{\bp^{\ell-1-j}\nabla^j}{ \xi  b^{rs}} U,_{rs}\big\|_{L^2(B^+_m)} \\
&\le C_\epsilon \|a\|_{H^\rk(\Omega)} \|\u^\epsilon\|_{H^{\ell+\frac{7}{8}}(\Omega^+)}\,.
\end{align*}
Therefore, for $0 \le j \le \ell -1$,  letting $\bp^{\ell-1-j}\nabla^j$ act on (\ref{pf_of_thm12.12_id2}),
\begin{equation}\label{pf_of_thm12.12_id3}
\xi  b^{\n\n} \bp^{\ell-1-j}\nabla^j U,_{\n\n} = \text{\bf\emph{G}}_{(\ell,j)} - \sum_{(r,s) \ne (\n,\n)} \xi  b^{rs} 
\bp^{\ell-1-j}\nabla^jU,_{rs}
\end{equation}
for a function $\text{\bf\emph{G}}_{(\ell,j)}$ satisfying
$$
\|\text{\bf\emph{G}}_{(\ell,j)}\|_{L^2(B^+_m)} \le C \Big[\|\f\|_{H^{\ell-1}(\Omega^+)} 
+ \|a\|_{H^\rk(\Omega^ + )} \|\u^\epsilon\|_{H^{\ell+\frac{7}{8}}(\Omega^+)} \Big]\,.
$$
Now we argue by induction on $0\le j \le \ell -1$. 
By (\ref{brs_elliptic}), 
$b^{\n\n} \ge \smallexp{$\frac{\lambda}{2}$}$ so that when $j=0$, the inequalities 
(\ref{vector-valued_elliptic_est_temp2}) and (\ref{pf_of_thm12.12_id3}) show that
\begin{align*}
& \big\| \xi  \bp^{\ell-1} U,_{\n\n}\big\|_{L^2(B^+_m)} \le  \|\text{\bf\emph{G}}_{(\ell,j)}\|_{L^2(B^+_m)} + \sum_{(r,s) \ne (\n,\n)}
 \|b^{rs} \|_{L^\infty(B^+_m)} \big\| \xi  \p^{\ell-1} U,_{rs}\big\|_{L^2(B^+_m)}
 \\
&\qquad \qquad \smallexp{$\displaystyle\qquad \le C \Big[ \|\f\|_{H^{\ell-1}(\Omega^+)} 
+ \|\g\|_{H^{\ell-0.5}(\Gamma)} + \|a\|_{H^\rk(\Omega^+)} \|\u^\epsilon\|_{H^{\ell+\frac{7}{8}}(\Omega^+)} \Big]$}
\end{align*}
which, combined with (\ref{vector-valued_elliptic_est_temp2}), provides the estimate
\begin{align*}
\big\| \xi  \bp^{\ell-1} \nabla^2 U\big\|_{L^2(B^+_m)} 
\le C \Big[ \|\f\|_{H^{\ell-1}(\Omega^+)} + \|\g\|_{H^{\ell-0.5}(\Gamma)} + \|a\|_{H^\rk(\Omega^+)} \|\u^\epsilon\|_{H^{\ell+\frac{7}{8}}(\Omega^+)} \Big]
\end{align*}
Repeating this process for $j=1,\cdots,\ell$ and including the analysis in the $-$-phase, we conclude that
\begin{align} 
 \big\| \xi \nabla^{\ell+1} U_\pm\big\|_{L^2(B^\pm_m)} 
 \le C \Big[ \|\f_\pm\|_{H^{\ell-1}(\Omega^\pm)} + \|\g\|_{H^{\ell-0.5}(\Gamma)} + \|a_\pm\|_{H^\rk(\Omega^\pm)} \|\u_\pm^\epsilon\|_{H^{\ell+\frac{7}{8}}(\Omega^\pm)} \Big]\,.
 \label{vector-valued_ellptic_bdy_est_with_Sobolev_coeff}
\end{align} 

\vspace{.1 in}
\noindent 
{\bf Step 3: Completing the regularity theory.}
Let $\chi_\pm \ge 0$ be in $\mC^ \infty _c(\Omega^\pm)$ so that $\supp(\chi_\pm) \cptsubset \Omega^\pm$. 
Repeating the computations above, we find that
\begin{equation}\label{vector-valued_ellptic_int_est_with_Sobolev_coeff}
 \|\chi_\pm \nabla^{\ell+1} \u_\pm^\epsilon\|_{L^2(\Omega^\pm)}
\le C \Big[ \|\f_\pm\|_{H^{\ell-1}(\Omega^\pm)} + \|a_\pm\|_{H^\rk(\Omega^\pm)} \|\u_\pm^\epsilon\|_{H^{\ell+\frac{7}{8}}(\Omega^\pm)} \Big] \,.
\end{equation}
The inequalities
(\ref{vector-valued_ellptic_bdy_est_with_Sobolev_coeff}) and (\ref{vector-valued_ellptic_int_est_with_Sobolev_coeff}) establishes the
inequality
\begin{equation}\label{vector-valued_ellptic_est_with_Sobolev_coeff}
\|\u_\pm^\epsilon\|_{H^{\ell+1}(\Omega^\pm)} \le C \Big[
\|\f_\pm\|_{H^{\ell-1}(\Omega^\pm)} + \|\g\|_{H^{\ell-0.5}(\Gamma)} 
+ \|a_\pm\|_{H^\rk(\Omega^\pm)} \|\u_\pm^\epsilon\|_{H^{\ell+\frac{7}{8}}(\Omega^\pm)} \Big] \,.
\end{equation}
Since
\begin{align*}
\|\u_\pm^\epsilon\|_{H^{\ell+\frac{7}{8}}(\Omega^\pm)} &\le C \|\u_\pm^\epsilon\|^{1-\frac{1}{8\ell}}_{H^{\ell+1}(\Omega^\pm)} 
\|\u_\pm^\epsilon\|^{\frac{1}{8\ell}}_{H^1(\Omega^\pm)} \,,
\end{align*}
 Young's inequality shows that
\begin{align*}
&\|\u_\pm^\epsilon\|_{H^{\ell+1}(\Omega^\pm)} \\
&\ \le C_\delta \Big[\|\f_\pm\|_{H^{\ell-1}(\Omega^\pm)} + \|\g\|_{H^{\ell-0.5}(\Gamma)} 
+ \P\big(\|a_\pm\|_{H^\rk(\Omega^\pm)}\big) \|\u_\pm^\epsilon\|_{H^1(\Omega^\pm)} \Big] 
+ \delta \|\u_\pm^\epsilon\|_{H^{\ell+1}(\Omega^\pm)}
\end{align*}
for some polynomial  function $\P$. Finally, the inequality  (\ref{vector-valued_elliptic_regularity})  is established by choosing $\delta > 0$
sufficiently small,  letting $\epsilon\to 0$,  and using the a priori  $H^1$-estimate.
\end{proof}

\appendix
\section{Some basic inequalities}
\begin{lemma}\label{prop:HkHl_product}
For $\rk > \novertwo$ and $0 \le \ell\le \rk$,
let $\rO\subseteq \Rn$ be a bounded smooth domain. Then for all $\epsilon \in \big(0,\smallexp{$\displaystyle{}\frac{1}{4}$}\big)$, there exists a constant $C_\epsilon$ depending on $\epsilon$ such that for all $f\in H^\rk(\rO)$ and $g\in H^{\ell- \epsilon }(\rO)$,
\begin{equation}\label{commutator_estimate_elliptic_est_temp}
\sum_{j=1}^\ell \|\nabla ^j f\nabla ^{\ell - j} g\|_{L^2(\rO)} \le C_\epsilon \|f\|_{H^\rk(\rO)} \|g\|_{H^{\ell-\epsilon}(\rO)} \,.
\end{equation}
\end{lemma}
\begin{proof}
We estimate $ \nabla ^j f  \nabla ^{\ell-j} g$ for $j=1,\cdots,\ell$ as follows:

\vspace{.05 in}
\noindent
{\bf Step 1.} If $1\le j\le \novertwo$, by the Sobolev inequalities
  \begin{align*}
  \|w\|_{L^\frac{\n}{j-\epsilon}(\rO)} \hspace{-2pt}&\le 
  C_\epsilon \|w\|_{H^{\frac{\n}{2} - j + \epsilon}(\rO)} \quad\text{(\hspace{1pt}if $0< \epsilon < 1$)}\,,\\
  \|w\|_{L^\frac{2\n}{\n-2(j-\epsilon)}(\rO)} \hspace{-2pt}&\le C \|w\|_{H^{j-\epsilon}(\rO)}\,,
  \end{align*}
  we find that
  \begin{align*}
  \|\nabla ^j f\nabla ^{\ell-j} g\|_{L^2(\rO)} &\le \|\nabla ^j f\|_{L^\frac{\n}{j-\epsilon}(\rO)} \|\nabla ^{\ell-j} g\|_{L^\frac{2\n}{\n-2(j-\epsilon)}(\rO)} \\
  &\le C_\epsilon \|f\|_{H^{\frac{\n}{2}+\epsilon}(\rO)} \|g\|_{H^{\ell-\epsilon}(\rO)} \,.
  \end{align*}

\vspace{.05 in}
\noindent
{\bf Step 2.}  If $j = \ell$, by the Sobolev inequality
  $$
  \qquad\qquad\quad \|w\|_{L^\infty(\rO)} \hspace{-3pt}\le C_\epsilon \|w\|_{H^{\frac{\n}{2} + \epsilon}(\rO)}\,,
  $$
  we find that
  $$
  \|\nabla ^j f  \nabla ^{\ell-j} g\|_{L^2(\rO)} 
  \le C_\epsilon \|f\|_{H^\ell(\rO)} \|g\|_{H^{\frac{\n}{2} + \epsilon}(\rO)} \,.
  $$

\vspace{.05 in}
\noindent
{\bf Step 3.}  If $\novertwo < j < \ell$ {\rm(}this happens only when $\novertwo < \ell \le \rk${\rm)}, we consider the following two sub-cases:

\vspace{.05 in}
\noindent
{\it Case A: $\ell \le \n$}\,: Similar to the previous case, by the Sobolev inequalities
  $$
  \|w\|_{L^\frac{2\n}{\n - 2(\ell-j)}(\rO)} \hspace{-3pt}\le C \|w\|_{H^{\ell-j}(\rO)} \ \ \text{and}\ \
  \|w\|_{L^\frac{\n}{\ell-j}(\rO)} \hspace{-3pt}\le C \|w\|_{H^{\frac{\n}{2} - \ell + j}(\rO)} \,,
  $$
 and hence,  we obtain that
  \begin{align}
   \| \nabla ^j f \nabla ^{\ell-j} g\|_{L^2(\rO)} &\le \|\nabla ^j f\|_{L^\frac{2\n}{n - 2(\ell-j)}(\rO)} \|\nabla ^{\ell-j} g\|_{L^\frac{\n}{\ell-j}(\rO)} 
  \le C \|f\|_{H^\ell(\rO)} \|g\|_{H^{\frac{\n}{2}}(\rO)}\,. \nonumber
  \end{align}
  
\vspace{.05 in}
\noindent
{\it Case B: $\n < \ell \le \rk$}\,: If $j > \rk- \novertwo$\,, by the Sobolev inequalities
  $$
  \|w\|_{L^\frac{2\n}{\n - 2(\rk-j)}(\rO)} \hspace{-3pt}\le C \|w\|_{H^{\rk-j}(\rO)}\ \ \text{and}\ \
  \|w\|_{L^\frac{\n}{\rk-j}(\rO)} \hspace{-3pt}\le C \|w\|_{H^{\frac{\n}{2} - \rk + j}(\rO)}\,,
  $$
  we obtain that
  \begin{align}
   \|\nabla ^j f \nabla ^{\ell-j} g\|_{L^2(\rO)} &\le \| \nabla ^j f\|_{L^\frac{2\n}{n - 2(\rk-j)}(\rO)} \| \nabla^{\ell-j} g\|_{L^\frac{\n}{\rk-j}(\rO)} 
  \le C \|f\|_{H^\rk(\rO)} \|g\|_{H^{\frac{\n}{2}-\rk+\ell}(\rO)}\,. \nonumber
  \end{align}
   Now suppose that $\novertwo < j \le \rk - \novertwo$\,. Note that if $0 < \epsilon < \smallexp{$\displaystyle{}\frac{1}{2}$}$,
  \begin{align*}
   \|w\|_{H^{\frac{\n}{2} + \epsilon}(\rO)} \hspace{-2pt}&\le C_\epsilon \|w\|_{W^{j,\infty}(\rO)} \le C_\epsilon \|w\|_{H^\rk(\rO)} \,, \\
   \|w\|_{H^{\frac{\n}{2}-\rk+\ell}(\rO)} \hspace{-2pt}&\le C \|w\|_{H^{\ell-j}(\rO)} \le C \|w\|_{H^{\ell-\epsilon}(\rO)}\,.
  \end{align*}
 Therefore, by the Gagliardo-Nirenberg-Sobolev interpolation inequality, we obtain that
  \begin{align*}
  \| \nabla ^j f \nabla ^{\ell-j} g\|_{L^2(\rO)} &\le \|f\|_{W^{j,\infty}(\rO)} \|g\|_{H^{\ell-j}(\rO)} \\
   & \le C_\epsilon \|f\|^{1-\alpha_j}_{H^{\frac{\n}{2}+\epsilon}(\rO)} \|f\|^{\alpha_j}_{H^\rk(\rO)} \|g\|^{\alpha_j}_{H^{\frac{\n}{2}-\rk+\ell}(\rO)} \|g\|^{1-\alpha_j}_{H^{\ell-\epsilon}(\rO)}
  \end{align*}
  for some  $\alpha_j \in (0,1)$; hence, by Young's inequality, 
  \begin{align*}
   \| \nabla ^j f \nabla ^{\ell-j}g\|_{L^2(\rO)} 
   \le C_\epsilon \Big[\|f\|_{H^{\frac{\n}{2}+\epsilon}(\rO)} \|g\|_{H^{\ell-\epsilon}(\rO)} + \|f\|_{H^\rk(\rO)} \|g\|_{H^{\frac{\n}{2}-\rk+\ell}(\rO)} \Big]\,.
  \end{align*}

Summing over $\ell$,
we conclude that for $0 < \epsilon < \smallexp{$\displaystyle{}\frac{1}{2}$}$,
\begin{align*}
 \sum_{j=1}^\ell \| \nabla ^j f \nabla ^{\ell-j} g\|_{L^2(\rO)} 
 \le \left\{\begin{array}{ll}
C_\epsilon \|f\|_{H^{\frac{\n}{2}+\epsilon}(\rO)} \|g\|_{H^{\ell - \epsilon}(\rO)} & \text{if $\ell \le \novertwo$}\,,\vspace{.2cm}\\
C_\epsilon \Big[\|f\|_{H^{\frac{\n}{2}+\epsilon}(\rO)} \|g\|_{H^{\ell-\epsilon}(\rO)} + \|f\|_{H^\rk(\rO)} \|g\|_{H^{\frac{\n}{2}+\epsilon}(\rO)} \Big] & \text{otherwise}\,.
\end{array}
\right.
\end{align*}
Estimate (\ref{commutator_estimate_elliptic_est_temp}) is then obtained from the fact that for 
all $\epsilon \in \big(0,\smallexp{$\displaystyle{}\frac{1}{4}$}\big)$, 
$$
\frac{\n}{2} + \epsilon \le \rk \quad\text{and}\quad
\frac{\n}{2} + \epsilon \le \ell - \epsilon \text{ \ if (in addition)\ }\ell > \frac{\n}{2}\,.
$$
\end{proof}

\begin{corollary}\label{lem:useful_lemma_with_Sobolev_class_coeff1} For any $m \in \{ 1, ..., K\}$, and 
for $F \in H^\rk(B^\pm_m)$ and $G = H^ {\ell- \epsilon }(B^\pm_m)$ with $0< \epsilon < 1/4$ and $1\le \ell \le \rk$,
  \begin{equation}\label{commutator_estimate_elliptic_est2}
   \big\|\comm{\bp^\ell}{F} G\big\|_{L^2(B^\pm_m)} \le C_\epsilon \|F\|_{H^\rk(B^\pm_m)} \|G\|_{H^{\ell-\epsilon}(B^\pm_m)}\,,
  \end{equation}
   where $\comm{\bp^\ell}{F} G = \bp^\ell(F G) - F \bp^\ell G$. 
\end{corollary}

\vspace{.1in}

\noindent {\bf Acknowledgments.} DC was supported by the Centre for Analysis and Nonlinear PDEs funded by the UK EPSRC grant EP/E03635X and the Scottish Funding Council. SS was supported by the National Science Foundation under grants DMS-1001850 and DMS-1301380,
and by the Royal Society Wolfson Merit Award.   We thank Peter Constantin for discussing this problem with us, and for suggesting improvements on an early draft.

\bibliography{vortexpatch}
\bibliographystyle{plain}

\end{document}